\documentclass{article}
\usepackage[a4paper, total={6.0in, 8.5in}]{geometry}
\usepackage[T1]{fontenc}
\usepackage{amsthm,amsmath,amssymb}
\usepackage{bm}
\usepackage{mathrsfs}
\usepackage[all]{xy}
\usepackage{enumerate}
\usepackage{dutchcal}
\usepackage{cite}
\usepackage{abstract}

\title{Fibred sets within a predicative\\ and constructive  effective topos}
\author{Cipriano Junior Cioffo, Maria Emilia Maietti, and Samuele Maschio}
\date{}

\newtheorem{thm}{Theorem}[section]
\newtheorem{lemma}[thm]{Lemma}
\newtheorem{prop}[thm]{Proposition}
\newtheorem{cor}[thm]{Corollary}
\theoremstyle{definition}
\newtheorem{dfn}[thm]{Definition}
\newtheorem{rmk}[thm]{Remark}

\newcommand{\CC}{\boldsymbol{\mathcal{C}}}

\newcommand{\op}{\ensuremath{^{\operatorname{op}}}}

\newcommand{\msf}[1]{\mathsf{#1}}

\newcommand{\JJ}{\mathbf{J}}
\newcommand{\II}{\mathbf{I}}

\newcommand{\CUpsi}{(\msf{C},[\msf{H}],\nu)}

\newcommand{\tar}{\mbox{$\widehat{ID_1}$}}
\newcommand{\noteps}{\mathrel{\!\not\mathrel{\,\overline{\varepsilon}\!}\,}}
\newcommand{\Set}{\mathbf{Set}^r}
\newcommand{\Coll}{\mathbf{Col}^r}
\newcommand{\eff}{{\mathbf{Eff} }}

\newcommand{\mtt}{\mbox{{\bf mTT}}}
\newcommand{\emtt}{\mbox{{\bf emTT}}}

\newcommand{\Prop}{{\mathbf{Prop}^r}}
\newcommand{\Props}{{\mathbf{Prop}^r_\mathbf{s}}}

\newcommand{\peff}{\mathbf{pEff}}
\newcommand{\peffpr}{\mathbf{pEff}_{prop}}
\newcommand{\peffprs}{\mathbf{pEff}_{props}}
\newcommand{\mf}{{\bf MF}}
\newcommand{\Aa}{\msf{A}}
\newcommand{\BB}{\msf{B}}
\newcommand{\RR}{\msf{R}}
\newcommand{\Ss}{\msf{S}}
\newcommand{\czf}{\mathbf{CZF}}
\newcommand{\Dep}{Dep_\peff}
\newcommand{\Depset}{Dep_{\peff_{set}}}
\newcommand{\ARp}{\msf{A},[\msf{R}]}
\newcommand{\bmR}{{\msf{R}}}
\newcommand{\hott}{\mathbf{HoTT}}

\usepackage[all]{xy}
\usepackage{cancel}
\usepackage{tikz-cd}

\DeclareFontFamily{OT1}{pzc}{}
\DeclareFontShape{OT1}{pzc}{m}{it}{<->s*[1.1]pzcmi7t}{}
\DeclareMathAlphabet{\mathpzc}{OT1}{pzc}{m}{it}

\begin{document}

\maketitle
\begin{abstract}
We describe the fibrational structure of sets within the predicative
	variant $\peff$  of Hyland's Effective Topos $\eff$  previously introduced in Feferman's predicative theory of non-iterative fixpoints $\tar$.  Our structural analysis  can be carried out in constructive and predicative variants of  $\eff$ 
	within extensions of  Aczel's Constructive Zermelo-Fraenkel Set Theory.
All this shows that the  full subcategory of discrete objects of Hyland's Effective topos $\eff$ contains already a fibred predicative topos validating the formal Church's thesis, even  when both are formalized in a constructive metatheory.

\end{abstract}
\tableofcontents


\section{Introduction}
This paper builds on the work of \cite{maietti_maschio_2021} by further contributing to a predicative and constructive development of elementary topos theory, drawing inspiration from \cite{hofstra2004relative, moerdijkpalmgren2002type, VDB, VDBI}.

More in detail, we  investigate the fibred  structure of sets of the  predicative version $\peff$
of  Hyland's ``Effective Topos'' in \cite{maietti_maschio_2021},
said "predicative" since $\peff$ is
formalized in Feferman's weak theory of inductive definitions $\tar$ \cite{feferman}.

Here, in addition, we also show that  our fibred $\peff$ can be formalized  in a constructive and predicative meta-theory, such as
 Aczel's $\czf +{\bf REA_{\bigcup}} +\bf RDC$,  so that inductive and coinductive predicates can be defined in its internal logic.

The ultimate goal of this investigation is to show that $\peff$ has a fibred categorical structure sufficient to model 
the predicative and constructive two-level  Minimalist Foundation, for short \mf, conceived in \cite{mtt} and finalized in \cite{m09},
and its extensions with  (co)-inductive predicates in \cite{mmr21,mmr22,mfps23,phdthesisSabelli}. In this way, $\peff$ would provide a categorical predicative and constructive universe where to exhibit the constructive contents of proofs 
performed in \mf\ and extensions in terms of programs.
All this would contribute to broadening further
the already wide range of applications of  Hyland's ``Effective Topos'' $\eff$ in \cite{eff}, whose introduction opened different research lines
connecting computability  to topos theory  with new applications
to logic, mathematics, and computer science.

A main reason for such a broad range of applications is that  $\eff$   provides a realizability interpretation of higher-order logic extending Kleene realizability for intuitionistic arithmetic and, hence, validating formal Church's thesis, (see \cite{vanOOsten}).
Thanks to the results in this paper, our $\peff$  would become a predicative realizability model for \mf\  extended with inductive and coinductive predicates and formal Church's thesis. 

A key aspect of  $\eff$ that will help to reach our goal is that  the structure of $\eff$ roughly resembles
the two-level structure of \mf\ in \cite{m09}, which is one of the peculiarities of \mf\ among the variety of foundations for mathematics (and it is crucial to show its minimality by comparing each foundation with the most appropriate level of \mf).

 More in detail, the two-level structure of  \mf\ in \cite{m09} consists of 
an {\it intensional level} that can be used as a base for a proof-assistant and extraction of computational contents from proofs,  an {\it extensional level} where to develop mathematics in a language close to the usual practice, and an interpretation
of the extensional level in the intensional one using a quotient model. Such a quotient model was
analyzed categorically in \cite{qu12, elqu} as an instance the \emph{elementary quotient completion of a Lawvere's elementary doctrine} introduced in {\it loc.cit} as a \emph{generalization of the well-known notion of exact completion on a lex category} in \cite{car,carboniceliamagno,carbonivitale}.

Analogously,  the structure of  $\eff$ is an instance of a quotient completion called
{\it tripos-to-topos construction} introduced in  \cite{tripos} (see \cite{uec,MR16} for a precise comparison with the notion of elementary quotient completion).
But, while $\eff$ enjoys an impredicative internal language that can be modeled by just using suitable Lawvere doctrines, called
\emph{triposes},  both levels of \mf\ are formulated as constructive and predicative dependent type theories which need to be modeled by using fibrations or their variants as in \cite{jacobsbook, awodeynaturalmodels, categorieswithfamilies, maiettimodular}. 

Even the predicative version $\peff$ of $\eff$ in \cite{maietti_maschio_2021}  was built as the elementary quotient completion of a categorical rendering of the interpretation of the intensional level of \mf\ within Feferman's $\tar$  described in \cite{IMMSt}.  What remains is to model the extensional level \emtt\ of \mf\ in $\peff$.
Now, as hyperdoctrines are just enough to model the first order logical part of \mf,  to model also the set-theoretic part of \mf, and in particular of \emtt, we need to enucleate the fibred structure of $\peff$ by analyzing its codomain fibration.

More in detail, since both levels of \mf\  distinguish among collections, sets, propositions (quantifying on arbitrary collections), and small propositions (quantifying only on sets), to model all of them
we consider an equivalent presentation of the codomain fibration over $\peff$ whose fibers are families of dependent collections, which are in turn equivalent to suitable
internal groupoid actions  in the sense of \cite{maclane2012sheaves,facetI}  within $\peff$. Then, we identify a subfibration of families of sets of the codomain fibration of $\peff$, intended to model respectively sets and collections of \mf, whose corresponding sub-fibrations of subobjects will be used to model  small propositions and propositions of \mf.
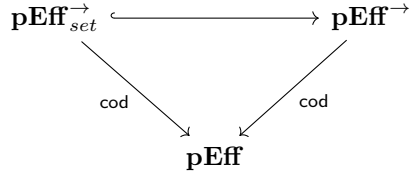
\[\begin{tikzcd}
		\\
		{{\mathbf{pEff}^\to_{set}}} && {{\mathbf{pEff}}^\to} \\
		\\
		& {\mathbf{pEff}}
		\arrow["{\mathsf{cod}}"', from=2-1, to=4-2]
		\arrow["{\mathsf{cod}}", from=2-3, to=4-2]
  \arrow[hook, from=2-1, to=2-3]
	\end{tikzcd}\]

Such a fibred structure can be shown to exists also when $\peff$ is formalized 
in $\czf$ + $\mathbf{REA}_{\bigcup} +\mathbf{RDC}$ by defining it as the elementary quotient completion of a categorical rendering of the realizability interpretation of the intensional level of \mf\ enriched with inductive and coinductive topological predicates in \cite{mmr21,mmr22}. So, $\peff$ would gain the expressive power to model both levels of \mf\ extended with inductive and coinductive predicates thanks to the results in \cite{mfps23, phdthesisSabelli}.

Furthermore, as already observed in \cite{maietti_maschio_2021}, our constructions of $\peff$ can be mapped into $\eff$, and actually in its subcategory of  the exact on lex completion of recursive objects,  called discrete objects in \cite{rosolinidiscrete,vanoostenhomotopy}, thanks
to the fact that $\eff$ can be viewed as the exact on lex completion
of the category of partitioned assemblies (see \cite{robinsonrosolini,chgenexco,maiettitrottageneralizedexistentialcompletionsregular}).

It is worth recalling that a  constructive predicative study of $\eff$ was developed in the context of algebraic set theory by B. van den Berg and I. Moerdijk (see  \cite{VDB})
by taking Aczel's $\czf$\ in \cite{czf} as the set theory to be realized in such toposes.  
The full subcategory of $\peff$ consisting of objects equipped small equivalence relations can be embedded into van den Berg and Moerdijk's predicative version of $\eff$,
when formalized in $\czf$ + $\mathbf{REA}_{\bigcup} +\mathbf{RDC}$.
Instead, the whole constructive and predicative $\peff$ can only be embedded into the category of discrete objects \cite{rosolinidiscrete,vanoostenhomotopy} in the same metatheory.

From these comparisons we conclude {\it the  full subcategory of discrete objects of Hyland's Effective topos $\eff$ }contains already a {\bf fibred predicative topos} validating the formal Church's thesis, even  when both are formalized in a constructive metatheory. 

As a future goal, recalling that $\eff$ was built as a first example of a non-Grothendieck topos sharing with localic toposes the structure of tripos-to-topos construction,
we hope to exploit a general notion of {\it fibred predicative  elementary topos} and a {\it predicative tripos-to-topos construction} to include predicative presentations of realizability and localic toposes.
Thanks to the results in this paper, our fibred  $\peff$ constitutes a key example for pursuing this project.

\section{Our predicative  meta-theory.}
\subsection{A classical predicative metatheory \` la Feferman:
$\tar$.}\label{tar}
Feferman's theory of non-iterative fixpoints $\tar$ \cite{feferman} is a first-order classical theory with equality containing $0$, $succ$, the functional symbols for the (definitions of) primitive recursive functions and a unary predicate symbol $P_{\varphi}$ for every admissible second-order formula $\varphi(x,X)$ (see \cite{feferman,MM2,maietti_maschio_2021} for more details). 

The axioms of $\tar$ include those of Peano arithmetic (even with those defining equations for primitive recursive functions) plus the \emph{induction principle} for every formula of $\tar$ and a \emph{fixpoint schema}: $P_{\varphi}(x)\leftrightarrow \varphi[P_{\varphi}/X]$ for every admissible second-order formula $\varphi$, where $\varphi[P_{\varphi}/X]$ is the result of substituting in $\varphi$ all the subformulas of the form $t\,\epsilon\, X$  with $P_{\varphi}(t)$.

We follow the notation for recursive functions in $\tar$ in \cite{maietti_maschio_2021}.

\subsection{A constructive and predicative meta-theory:  extensions of {\bf CZF}}
P. Aczel introduced $\mathbf{CZF}$ as a constructive theory of sets (see \cite{czf} for its axioms). In this paper we are interested in two extensions of it used in \cite{mmr21} and \cite{mmr22} to prove the consistency of the Minimalist Foundation with inductive and coinductive topological generation, respectively. The first extension adds to $\mathbf{CZF}$ the so-called regular extension axiom $\mathbf{REA}$ which states that every set is a subset of a regular set. The second extension adds to $\mathbf{CZF}$ the axiom $\bigcup-\mathbf{REA}$, stating that every set is contained in a $\bigcup$-regular set, and the axiom of relativized dependent choice ($\mathbf{RDC}$). In \cite{mmr21} and \cite{mmr22} one can see how these axioms play a crucial role in interpreting universes and inductive and coinductive topological generation. Further details on these theories can be found in \cite{czf}.

In what follows we will write $\mathcal{T}$ as a metavariable for any of the theories $\tar$, $\mathbf{CZF}+\mathbf{REA}$ or $\mathbf{CZF}+\mathbf{RDC}+\bigcup-\mathbf{REA}$.

\section{The starting structure of realized collections, sets, and propositions}
From \cite{maietti_maschio_2021}, we recall here the definition of the category of realized collections called  $\CC_r$ together with four indexed categories defined over it: that of realized collections and its sub-indexed category of realized sets
\[
\begin{tikzcd}[row sep=small]
	{\CC_{r}\op} \\
	&& {\mathbf{Cat}} \\
	{\CC_{r}\op}
	\arrow[""{name=0, anchor=center, inner sep=0}, "{\mathbf{Set}^r}"{pos=0.3}, from=1-1, to=2-3]
	\arrow["id"', from=1-1, to=3-1]
	\arrow[""{name=1, anchor=center, inner sep=0}, "{\mathbf{Col}^r}"'{pos=0.3}, from=3-1, to=2-3]
	\arrow[shorten <=4pt, shorten >=4pt, hook, from=0, to=1]
\end{tikzcd}\]
and  the doctrine of realized proposition and its sub-doctrine of small propositions
\[
\begin{tikzcd}[row sep=small]
	{\CC_{r}\op} \\
	&& {\mathbf{Pos}} \\
	{\CC_{r}\op}
	\arrow[""{name=0, anchor=center, inner sep=0}, "{\overline{\Props}}"{pos=0.3}, from=1-1, to=2-3]
	\arrow["id"', from=1-1, to=3-1]
	\arrow[""{name=1, anchor=center, inner sep=0}, "{\overline{\Prop}}"'{pos=0.3}, from=3-1, to=2-3]
	\arrow[shorten <=4pt, shorten >=4pt, hook, from=0, to=1]
\end{tikzcd}\]

Then, we consider the predicative effective topos $\peff$ in \cite{maietti_maschio_2021} defined as the base $\mathcal{Q}_{\overline{\mathbf{Prop}^r}}$ of the  elementary quotient completion of $\overline{\Prop} $. 
$$\peff :\equiv (\mathcal{Q}_{\overline{\mathbf{Prop}^r}})$$

 The key point of this paper is to show {\it  how to lift the indexed category  $ \mathbf{Set}^r $, and consequently also $\overline{\Props}$,  on  the predicative effective topos $\peff$}
to the aim of  investigating their categorical properties.  We will do this by switching from indexed categories to fibrations.

\subsection{The category $\CC_r$ of realized collections of $\mathcal{T}$.}
\label{reacol}

A \emph{realized collection} (or simply a \emph{class}) $\mathsf{A}$ of $\mathcal{T}$ is a formal expression $\{x|\,\varphi_{\mathsf{A}}\}$ representing a subclass of natural numbers in $\mathcal{T}$\footnote{In the set theoretical case, this means that $\varphi_{\mathsf{A}}(x)\vdash_{\mathcal{T}}Nat(x)$ where $Nat(x)$ is the formula expressing the fact that $x$ is a natural number in $\mathcal{T}$.} where $\varphi_{\mathsf{A}}$ is a formula of $\mathcal{T}$ with at most $x$ as free variable. 
We write $x\,\varepsilon\, \mathsf{A}$ as an abbreviation for $\varphi_{\mathsf{A}}$. Classes with provably equivalent membership relations in $\mathcal{T}$ are identified. An \emph{operation} between classes of $\mathcal{T}$ from $\mathsf{A}$ to $\mathsf{B}$ is an equivalence class $[\mathbf{n}]_{\approx_{\mathsf{A},\mathsf{B}}}$ of numerals with
$x\,\varepsilon\,\mathsf{A}\vdash\{\mathbf{n}\}(x)\,\varepsilon\,\mathsf{B}$ in $\mathcal{T}$. For such $\mathbf{n}$ and $\mathbf{m}$ we define $\mathbf{n}\approx_{\mathsf{A},\mathsf{B}} \mathbf{m}$ as $x\,\varepsilon \,\mathsf{A}\vdash \{\mathbf{n}\}(x)=\{\mathbf{m}\}(x)$ in $\mathcal{T}$. 

If $[\mathbf{n}]_{\approx_{\mathsf{A},\mathsf{B}}}:\mathsf{A}\rightarrow \mathsf{B}$ and $[\mathbf{m}]_{\approx_{\mathsf{B},\mathsf{C}}}:\mathsf{B}\rightarrow \mathsf{C}$ are operations between classes of $\mathcal{T}$, then their \emph{composition} is defined by $[\mathbf{m}]_{\approx_{\mathsf{B},\mathsf{C}}}\circ [\mathbf{n}]_{\approx_{\mathsf{A},\mathsf{B}}}:=[\Lambda x. \{\mathbf{m}\}(\{\mathbf{n}\}(x))]_{\approx_{\mathsf{A},\mathsf{C}}}$. 

The \emph{identity} operation for the class $\mathsf{A}$ of $\mathcal{T}$ is defined as $\mathsf{id}_{\mathsf{A}}:=[\Lambda x.x]_{\approx_{\mathsf{A},\mathsf{A}}}$.

\begin{dfn}\label{dfn: C_r}\normalfont
We denote with $\CC_{r}$ the category whose objects are realized collections of $\mathcal{T}$ and arrows are operations in $\mathcal{T}$ between them with their composition and identity operations. 
\end{dfn}
 The following result appears in  \cite[Theorem 4.3]{maietti_maschio_2021}.
 
\begin{thm}\label{thm: properties of C_r} $\CC_{r}$ is a finitely complete category with disjoint stable finite coproducts, parameterized list objects, and weak exponentials. Moreover, finite coproducts in $\CC_{r}$ are disjoint and stable.
\qed
\end{thm}

\subsection{Families of realized collections of $\mathcal{T}$.}
Here, we provide  a functorial account of the slice pseudofunctor over the 
the category of collection $\CC_{r}$ by defining an indexed category
\begin{equation}
\mathbf{Col}^{r}:\CC_{r}^{op}\rightarrow \mathbf{Cat}.
\end{equation}
as follows.

Suppose $\mathsf{A}$ is an object of $\CC_{r}$. A \emph{family of realized collections} $\mathsf{C}$ on $\mathsf{A}$ is a formal expression $\{x'|\,\varphi_{\mathsf{C}}\}$ representing a subclass of natural numbers depending on natural numbers in $\mathcal{T}$\footnote{In the set theoretical case, this means that $\varphi_{C}\vdash_{\mathcal{T}}Nat(x)\wedge Nat(x')$.} where $\varphi_{\mathsf{C}}$ is a formula of $\mathcal{T}$ with at most $x$ and $x'$ as free variables (we write $x'\,\varepsilon\, \mathsf{C}$ as an abbreviation for $\varphi_{\mathsf{C}}$) for which $x'\,\varepsilon\,\mathsf{C}\,\vdash\,x\,\varepsilon\, \mathsf{A}$ in $\mathcal{T}$. Families of realized collections on $\mathsf{A}$ with provably (in $\mathcal{T}$) equivalent membership relations are identified.

An \emph{operation} from a family of realized collections $\mathsf{C}$ on $\mathsf{A}$ to another $\mathsf{D}$ is given by an equivalence class $[\mathbf{n}]_{\approx_{\mathsf{C},\mathsf{D}}}$ of numerals such that in $\mathcal{T}$ we have 
$x'\,\varepsilon\,\mathsf{C}\,\vdash\,\{\mathbf{n}\}(x,x')\,\varepsilon\,\mathsf{D}$
with respect to the equivalence relation defined as follows: $\mathbf{n}\approx_{\mathsf{C},\mathsf{D}} \mathbf{m}$ if and only if $x'\,\varepsilon\, \mathsf{C}\,\vdash_{\mathcal{T}}\,\{\mathbf{n}\}(x,x')=\{\mathbf{m}\}(x,x')$.

 If $\mathsf{f}=[\mathbf{n}]_{\approx_{\mathsf{C},\mathsf{D}}}:\mathsf{C}\rightarrow \mathsf{D}$ and $\mathsf{g}=[\mathbf{m}]_{\approx_{\mathsf{D},\mathsf{E}}}:\mathsf{D}\rightarrow \mathsf{E}$ are operations between families of realized collection on $\mathsf{A}$, then their \emph{composition} is defined as 
 $$\mathsf{g}\circ \mathsf{f}:=[\Lambda x.\Lambda x'.\{\mathbf{m}\}(x,\{\mathbf{n}\}(x,x'))]_{\approx_{\mathsf{C},\mathsf{E}}}:\mathsf{C}\rightarrow\mathsf{E}\footnote{We will omit the subscripts of $\approx$ when they will be clear from the context.}.$$
If $\mathsf{C}$ is a family of realized collections on $\mathsf{A}$, then its \emph{identity} operation is defined by
 $\mathsf{id}_{\mathsf{C}}:=[\Lambda x.\Lambda x'.x']_{\approx_{\mathsf{C},\mathsf{C}}}:\mathsf{C}\rightarrow\mathsf{C}$.

 As discussed in \cite[Theorem 4.6]{maietti_maschio_2021}, for every object $\mathsf{A}$ of $\CC_{r}$, families of realized collections on $\mathsf{A}$ and operations between them together with their composition and identity operations define a category denoted with $\mathbf{Col}^{r}(\mathsf{A})$.

If $\mathsf{f}:=[\mathbf{n}]_{\approx}:\mathsf{A}\rightarrow \mathsf{B}$, then the following assignments give rise to a functor $\mathbf{Col}^{r}_{\mathsf{f}}$ from $\mathbf{Col}^{r}(\mathsf{B})$ to $\mathbf{Col}^{r}(\mathsf{A})$:
\begin{enumerate}
\item for every object $\mathsf{C}$ of $\mathbf{Col}^{r}(\mathsf{B})$, $\mathbf{Col}^{r}_{\mathsf{f}}(\mathsf{C}):=\{x'|\,x\,\varepsilon\, \mathsf{A}\ \wedge\  x'\,\varepsilon\, \mathsf{C}[\{\mathbf{n}\}(x))/x]\ \}$;
\item for every arrow $\mathsf{g}:=[\mathbf{k}]_{\approx}:\mathsf{C}\rightarrow\mathsf{D}$ of $\mathbf{Col}^{r}(\mathsf{B})$,
$$\mathbf{Col}^{r}_{\mathsf{f}}(\mathsf{g}):=[\Lambda x.\Lambda x'. \{\mathbf{k}\}(\{\mathbf{n}\}(x),x')]_{\approx}:{\mathbf{Col}^{r}_{\mathsf{f}}(\mathsf{C})\rightarrow \mathbf{Col}^{r}_{\mathsf{f}}(\mathsf{D})}.$$\end{enumerate}

Moreover, the assignments $\mathsf{A}\mapsto \mathbf{Col}^{r}(\mathsf{A})$ and $\mathsf{f}\mapsto \mathbf{Col}^{r}_{\mathsf{f}}$ define an indexed category

\begin{equation}
\mathbf{Col}^{r}:\CC_{r}^{op}\rightarrow \mathbf{Cat}.
\end{equation}

The functors $\mathbf{Col}^{r}_{\mathsf{f}}$ are also called \emph{substitution functors} and, as proven in \cite[Theorem 4.8]{maietti_maschio_2021}, these functors have left adjoints
\begin{equation}\label{equation: left adjoint Col}
\Sigma_{\mathsf{f}}:\mathbf{Col}^{r}(\mathsf{A})\rightarrow \mathbf{Col}^{r}(\mathsf{B})
\end{equation}
which sends an object $\msf{C}\in\Coll(\Aa)$ to $$\Sigma_{\msf{f}}(\msf{C}):=\{x'|x=\{\mathbf{n}\}(p_1(x')) \wedge \ p_2(x')\ \varepsilon\  \mathsf{C}[p_1(x')/x]\}$$
and an arrow $\msf{g}:=[\mathbf{m}]_{\approx}:\msf{C}\to\msf{D}$ to $$\Sigma_{\msf{f}}(\msf{g}):=[\Lambda x.\Lambda x'. \{\mathbf{pair}\}(\{\mathbf{p}_1\}(x'),\{\mathbf{m}\}(\{\mathbf{p}_1\}(x'),\{\mathbf{p}_2\}(x')))]_{\approx}$$

Furthermore, in \cite[Theorem 4.9]{maietti_maschio_2021} the authors provide also a weak versions of right adjoints
\begin{equation}\label{equation: right adjoint Col}
\Pi_{\mathsf{f}}:\mathbf{Col}^{r}(\mathsf{A})\rightarrow \mathbf{Col}^{r}(\mathsf{B})
\end{equation}

Now we recall an equivalent presentation of families of realized collections, the following result appears as \cite[Theorem 4.12]{maietti_maschio_2021}.

\begin{thm}\label{thm: equivalence Col(A) e Cr/A}
    For every $\mathsf{A}$ in $\CC_{r}$,  $\CC_{r}/\mathsf{A}$ and $\mathbf{Col}^{r}(\mathsf{A})$ are equivalent.
    \qed
\end{thm}
Therefore $\mathbf{Col}^{r}$ is a functorial account of the slice pseudofunctor on $\CC_{r}$ and hence it provides a split fibration
\[\pi:\mathsf{Gr}(\Coll)\to \CC_{r}\]


\subsection{Families of realized sets}  We recall from \cite[\S 4.3]{maietti_maschio_2021} that it is possible to define internally in $\CC_{r}$ a universe of sets as a fixpoint of a suitable admissible formula of $\mathcal{T}$.  The universe of sets in $\CC_{r}$ is defined as $\msf{U}_{\mathsf{S}}:=\{x|\;\mathsf{Set}(x)\, \}$ following \cite{mmr21,mmr22} when the meta-theory $\mathcal{T}$ is an extension of $\mathbf{CZF}$.
Instead, when $\mathcal{T}$ is $\tar$, following
\cite{IMMSt}, the universe of sets  is defined in a more complex way as 
$\msf{U}_{\mathsf{S}}:=\{x|\;\mathsf{Set}(x)\,\wedge\,\forall t\,(t\,\overline{\varepsilon}\,x\leftrightarrow \neg\, t\noteps x)\}$  
 with the help of a membership relation $\,\overline{\varepsilon}\,x\ $ and its negation $\neg\, t\noteps x$ defined simultaneously all together (recall that $\tar $ is an arithmetical theory).

\begin{dfn}\label{dfn: Set(A)}\normalfont
Let $\mathsf{A}$ be an object of $\CC_{r}$. The category $\mathbf{Set}^{r}(\mathsf{A})$ is the full subcategory of $\mathbf{Col}^{r}(\mathsf{A})$ 
\[\begin{tikzcd}
	{\Set(\mathsf{A})} & {\Coll(\mathsf{A})} 
	\arrow[hook, from=1-1, to=1-2]
 \end{tikzcd}\]
whose objects, called \emph{families of realized sets} on $\mathsf{A}$, are families of realized collections on $\mathsf{A}$ of the form $$\tau_{\mathsf{A}}(\mathbf{n}):=\{x'|\,x'\,\overline{\varepsilon}\, \{\mathbf{n}\}(x)\,\wedge\, x\,\varepsilon \,\mathsf{A}\}$$ for a numeral $\mathbf{n}$ defining an operation 
 $[\mathbf{n}]_{\approx}:\mathsf{A}\rightarrow \mathsf{U_{S}}$ in $\CC_{r}$. 
\end{dfn}

Applying the Grothendieck construction we obtain the diagram

\[\begin{tikzcd}
	{{\mathsf{Gr}(\mathbf{Set}^r)}} && {{\mathsf{Gr}(\mathbf{Col}^r)}}  \\
	& {\CC_{r}}
	\arrow[hook, from=1-1, to=1-3]
	\arrow["{\pi}"', from=1-1, to=2-2]
	\arrow["{\pi}", from=1-3, to=2-2]
\end{tikzcd}\]


\subsection{Categorical properties of realized collections and sets}

We recall the properties of the categories of families of realized collections $\mathbf{Col}^r(\msf{A})$. The following result appears as \cite[Theorem 4.7]{maietti_maschio_2021}.

\begin{thm}\label{thm: properties of Col}
For every object $\mathsf{A}$ of $\CC_{r}$, $\mathbf{Col}^{r}(\mathsf{A})$ is a finitely complete category with parameterized list objects, finite coproducts, and weak exponentials. Moreover for every arrow $\mathsf{f}:\mathsf{A}\rightarrow\mathsf{B}$ in $\CC_{r}$, the functor $\mathbf{Col}^{r}_{\mathsf{f}}$ preserves this structure. \qed
    \end{thm}

We recall the main ingredients of the proof of Theorem \ref{thm: equivalence Col(A) e Cr/A} which will be used in the rest of the paper.
The proof of the above result relies on the definition of the following functors. The first one is 
\begin{equation}\label{equation: functor J}
    \mathbf{J}_{\mathsf{A}}:\CC_{r}/\mathsf{A}\rightarrow \mathbf{Col}^{r}(\mathsf{A})
\end{equation}
 defined by  the assignments
 $$\begin{array}{rcl} [\mathbf{b}]_{\approx}:\mathsf{B}\rightarrow \mathsf{A} & \mapsto\ &  \{x'|\,x'\,\varepsilon\, \mathsf{B}\wedge x=\{\mathbf{b}\}(x')\}\\[3pt]
 [\mathbf{n}]_{\approx}\ \  &\ \mapsto\ &\ [\Lambda x.\Lambda x'.\{\mathbf{n}\}(x')]_{\approx}. \end{array}$$

The second is the functor 
\begin{equation}\label{equation: functor I}
\mathbf{I}_{\mathsf{A}}:\mathbf{Col}^{r}(\mathsf{A})\rightarrow\CC_{r}/\mathsf{A}
\end{equation}
 defined by the assignments 
$$\begin{array}{rcl}\mathsf{B}\ \ & \mapsto\ & \ \mathsf{p}_{1}^{\Sigma}:=[\mathbf{p}_{1}]_{\approx}:\Sigma(\mathsf{A},\mathsf{B})\rightarrow \mathsf{A}\\[3pt]
[\mathbf{n}]_{\approx}\ \ & \mapsto\ & \ [\Lambda x.\{\mathbf{pair}\}(\{\mathbf{p}_{1}\}(x),\{\mathbf{n}\}(\{\mathbf{p}_{1}\}(x),\{\mathbf{p}_{2}\}(x)))]_{\approx}.
\end{array}$$

where $\Sigma(\mathsf{A},\mathsf{B})$ is the object defined as 
\begin{equation}\label{equation : Sigma}
    \Sigma(\mathsf{A},\mathsf{B}):= \{x|\,p_{1}(x)\,\varepsilon\, \mathsf{A}\wedge p_{2}(x)\,\varepsilon\, \mathsf{B}[p_{1}(x)/x])\}.
\end{equation}

The above Theorem implies the following result which appears as \cite[Corollary 4.13]{maietti_maschio_2021}.

\begin{cor}\label{cor: Cr weakly lccc}
$\CC_r$ is \emph{weakly locally cartesian closed}.
\end{cor}
Observe that, since every arrow in $\CC_r$  can be represented by many different numerals, the category 
$\CC_r$ is weakly locally cartesian closed but not cartesian closed. 

Now we recall the properties of the category of families of realized sets. 


\begin{dfn}\label{dfn: representable}\normalfont
   An arrow $\msf{f}:\mathsf{A}\rightarrow \mathsf{B}$ in $\CC_{r}$ is \emph{set-fibred} if there exists an object $\mathsf{C}$ in $\mathbf{Set}^{r}(\mathsf{B})$ for which $\msf{f}$ is isomorphic in $\CC_{r}/\mathsf{B}$ to $\mathbf{I}_{\mathsf{B}}(\mathsf{C})$.
\end{dfn}


\begin{lemma}\label{rmk: I representable arrow}\normalfont
If $\msf{f}:\mathsf{B}\rightarrow \mathsf{C}$ is an arrow in $ \Set(\mathsf{A})$, then $\mathbf{I}_{\Aa}(\msf{f})$ is a set-fibred arrow.
\begin{proof}
    Take the object $\msf{D}$ in $\Set(\Sigma(\Aa,\msf{C}))$ defined as
    \[\msf{D}:=\{y\ |\ p_1(y)\ \varepsilon\ \msf{B}[p_1(x)/x]\ \wedge p_2(y)=p_2(x) \wedge  p_2(y)=\{\mathbf{n}_\msf{f}\}(p_1(x),p_1(y))\}\]
    then $\II_{\Aa}(\msf{f})$, seen as an object in $\CC_r/(\Sigma(\Aa,\msf{C})$, is isomorphic to $\II_{\Sigma(\Aa,\msf{C})}(\msf{D})$.
\end{proof}
\end{lemma}

 Using the encoding of $\Sigma$ and $\Pi$ we obtain the following properties about substitution along set-fibred maps.

\begin{lemma}\label{lemma: repesentable left adjoint}
    If $\msf{f}:\msf{A}\to \msf{B}$ is set-fibred then  $\Set_\msf{f}$  has a left adjoint.
    \begin{proof}
        Consider the arrow $\mathsf{p}_{1}^{\Sigma}=[\mathbf{p}_{1}]_{\approx}:\Sigma(\mathsf{B},\tau_{\mathsf{B}}(\mathbf{n}))\rightarrow \mathsf{B}$. Then 
$$\Sigma_{\mathsf{p}_{1}^{\Sigma}}(\tau_{\Sigma(\mathsf{B},\tau_{\mathsf{B}}(\mathbf{n}))}(\mathbf{m}))\simeq \tau_{\mathsf{B}}(\Lambda x. \{\boldsymbol{\sigma}\}( \{\mathbf{n}\}(x),\Lambda y. \{\mathbf{m}\}(\{\mathbf{pair}\}(x,y))))$$
where $\boldsymbol{\sigma}$ is a numeral representing the construction of codes of $\Sigma$-types in the universe of realized sets \cite{IMMSt}.
    \end{proof}
\end{lemma}
\begin{lemma}\label{lemma: representable right adjoint}
    If $\mathsf{f}:\mathsf{A}\rightarrow \mathsf{B}$ in $\CC_{r}$ is set-fibred,  for every object $\mathsf{C}$ in $\mathbf{Set}^{r}(\mathsf{A})$ the object $\Pi^{r}_{\mathsf{f}}(\mathsf{C})$ is isomorphic in $\mathbf{Col}^{r}(\mathsf{B})$ to an object of $\mathbf{Set}^{r}(\mathsf{B})$.
    \begin{proof}
Consider the arrow $\mathsf{p}_{1}^{\Sigma}=[\mathbf{p}_{1}]_{\approx}:\Sigma(\mathsf{B},\tau_{\mathsf{B}}(\mathbf{n}))\rightarrow \mathsf{B}$. Then, considering the definition of $\Pi^r_{\msf{f}}$ in \cite[Theorem 4.9]{maietti_maschio_2021}, we get that
$$\Pi_{\mathsf{p}_{1}^{\Sigma}}(\tau_{\Sigma(\mathsf{B},\tau_{\mathsf{B}}(\mathbf{n}))}(\mathbf{m}))\simeq \tau_{\mathsf{B}}(\Lambda x. \{\boldsymbol{\pi}\}( \{\mathbf{n}\}(x),\Lambda y. \{\mathbf{m}\}(\{\mathbf{pair}\}(x,y))))$$
where $\boldsymbol{\pi}$ is a numeral representing the construction of codes of $\Pi$-types in the universe of realized sets \cite{IMMSt}.
\end{proof}
\end{lemma}

Using set-fibred arrows, we obtain the following extension of \cite[Theorem 4.24]{maietti_maschio_2021}.

\begin{thm}\label{thm: dependent collections}
For every $\mathsf{A}$ in $\CC_{r}$, $\mathbf{Set}^{r}(\mathsf{A})$ is a finitely complete category with disjoint stable finite coproducts, parameterized list objects and weak exponentials. 
Moreover, if $\mathsf{f}:\mathsf{A}\rightarrow \mathsf{B}$ is set-fibred, then $\mathbf{Set}^{r}_{\mathsf{f}}$ has a left adjoint and for every object $\mathsf{C}$ in $\mathbf{Set}^{r}(\mathsf{A})$, $\Pi^{r}_{\mathsf{f}}(\mathsf{C})$ is isomorphic in $\mathbf{Col}^{r}(\mathsf{B})$ to an object of $\mathbf{Set}^{r}(\mathsf{B})$.

The embedding of $\mathbf{Set}^{r}$ in $\mathbf{Col}^{r}$ preserves finite limits, finite coproducts,  parameterized list objects, and weak exponentials. 
\qed\end{thm}  

\subsection{Realized propositions.}

Now we recall the notion of realized propositions on $\CC_{r}$.

\begin{dfn}\label{dfn: Prop^r}\normalfont
The indexed category $\overline{\Prop}:\CC_{r}^{op}\rightarrow \mathbf{Pos}$ is defined as the poset reflection of $\mathbf{Col}^{r}$.
In this case we write $x'\Vdash \mathsf{P}$ instead of $x'\,\varepsilon\, \mathsf{P}$ for an object $\mathsf{P}$ in $\mathbf{Col}^{r}(\mathsf{A})$ seen as a representative of an object $[\msf{P}]$ of $\overline{\Prop}(\mathsf{A})$.
\end{dfn}


We recall the following definition:
\begin{dfn}\normalfont
If $\CC$ is a finitely complete category, the doctrine of its weak subobjects $\mathbf{wSub}_{\CC}$ is the posetal reflection of the slice pseudofunctor associated with $\CC$.
\end{dfn}

As proved in \cite[Theorem 4.19]{maietti_maschio_2021}, realized propositions correspond to weak subobjects.

\begin{thm}\label{weaksub}
   The functor $\overline{\Prop}$ is a first-order hyperdoctrine naturally isomorphic to the doctrine of weak subobjects $\mathbf{wSub}_{\CC_r}$ of $\CC_r$.
   \qed
\end{thm}

If $\msf{f}:\Aa\to \BB$ is an arrow in $\CC_r$, the left and right adjoint to $\overline{\Prop}_\msf{f}$ will be denoted with $\exists_{\msf{f}}$ and $\forall_{\msf{f}}$. respectively.

\subsection{Realized small propositions.}

\begin{dfn}\label{dfn: dependent set}\normalfont The functor $\overline{\Props}:\CC_{r}^{op}\rightarrow \mathbf{Pos}$ is defined as the poset reflection of the indexed category $\mathbf{Set}^{r}$. Then, fibre objects of $\overline{{\Props}}$ are called {\it small realized propositions}\footnote{Note that here we use the same notation of \cite{maietti_maschio_2021} for the indexed categories $\mathbf{Set}^{r}$ and $\overline{{\Props}}$, that were defined in \emph{loc.\ cit.}\ on the subcategory $\mathcal{S}_r$ of $\CC_r$ of \emph{realized sets}.}.
\end{dfn}
We now obtain the following theorem {which is a straightforward extension of} \cite[Theorem 4.27]{maietti_maschio_2021}.

\begin{thm}\label{thm: small Prop^r properties} 
For every object $\mathsf{A}$ in $\CC_{r}$, $\overline{\Props}(\mathsf{A})$ is a Heyting algebra
and for every arrow $\mathsf{f}$ in $\CC_{r}$, $(\overline{\Props})_{\mathsf{f}}$ is a homomorphism of Heyting algebras.

Moreover, if  $\mathsf{f}:\mathsf{A}\rightarrow \mathsf{B}$ in $\CC_{r}$ is set-fibred, then $(\overline{\Props})_{\mathsf{f}}$ has left $\exists_\msf{f}$ and right  $\forall_\msf{f}$ adjoints satisfying Beck-Chevalley conditions.
\qed
\end{thm}

\subsection{Further notation.} For every $\mathsf{f}:=[\mathbf{n}]_{\approx}:\mathsf{A}\rightarrow \mathsf{B}$ in $\CC_{r}$ and every $\mathsf{C}$


in $\mathbf{Col}^{r}(\mathsf{B})$, we denote with $\Sigma(\mathsf{f},\mathsf{C})$ the arrow $[\Lambda x. \{\mathbf{pair}\}(\{\mathbf{n}\}(\{\mathbf{p}_{1}\}(x)),\{\mathbf{p}_{2}\}(x))\, ]_{\approx}$ from $\Sigma(\mathsf{A},\mathbf{Col}_{\msf{f}}^r(\mathsf{C}))$ to $\Sigma(\mathsf{B},\mathsf{C})$. This arrow makes the following diagram a pullback in 
$\CC_{r}$

$$\begin{tikzcd}
	{\Sigma(\mathsf{A},\mathbf{Col}_{\mathsf{f}}^{r}(\mathsf{C}))} && {\Sigma(\mathsf{B},\mathsf{C})} \\
	{\mathsf{A}} && {\mathsf{B}}
	\arrow["{\Sigma(\mathsf{f},\mathsf{C})}", from=1-1, to=1-3]
	\arrow["{\mathsf{p}_{1}^{\Sigma}}"', from=1-1, to=2-1]
	\arrow["{\mathsf{p}_{1}^{\Sigma}}", from=1-3, to=2-3]
	\arrow["{\mathsf{f}}"', from=2-1, to=2-3]
\end{tikzcd}$$
for details see \cite[Lemma 4.11]{maietti_maschio_2021}.

 Now let $\msf{A}$ be an object in $\CC_r$ and let $\msf{B}$ be an object of $\mathbf{Col}^r(\msf{A})$. If $\msf{C}$ is an object in $\mathbf{Col}^r(\Sigma(\msf{A}, \msf{B}))$, then we define 
 $\Sigma^{\msf{A}}(\msf{B},\msf{C})$ as the object in $\mathbf{Col}^r(\msf{A})$ given by
 $\Sigma_{\msf{p}^\Sigma_1}(\msf{C})$.

 Moreover, for every arrow $\msf{f}: \msf{B}'\to\msf{B}$ in $\textbf{Col}^r(\msf{A})$, we define  $\Sigma^{\msf{A}}(\msf{f}, \msf{C})$ as the arrow $$\Lambda x. \Lambda x'. \{\mathbf{pair}\} (\{\mathbf{pair}\}(x,\{\mathbf{n}_\msf{f}\}(x,\{\mathbf{p_2}\}(\{\mathbf{p}_1\}(x'))),\{\mathbf{p_2}\}(x') )$$
 which makes the following diagram a pullback in $\Coll(\mathsf{A})$

 $$
 \begin{tikzcd}
 	{\Sigma^\mathsf{A}(\mathsf{B}',\mathbf{Col}_{\mathsf{f}}^{r}(\mathsf{C}))} && {\Sigma^\mathsf{A}(\mathsf{B},\mathsf{C})} \\
 	{\mathsf{B}'} && {\mathsf{B}}
 	\arrow["{\Sigma^\mathsf{A}(\mathsf{f},\mathsf{C})}", from=1-1, to=1-3]
 	\arrow["{\mathsf{p}_{1}^{\mathsf{A},\Sigma}}"', from=1-1, to=2-1]
 	\arrow["{\mathsf{p}_{1}^{\mathsf{A},\Sigma}}", from=1-3, to=2-3]
 	\arrow["{\mathsf{f}}"', from=2-1, to=2-3]
 \end{tikzcd}$$

where $\msf{p}_1^{\msf{A},\Sigma}:=[\Lambda x.\Lambda x'. \{\mathbf{p_2}\}(\{\mathbf{p_1}\}(x'))]_\approx$.

We now define a \emph{currying} isomorphism  in $\CC_r$
for every object $\msf{A}\in \CC_r$,  $\msf{B}\in \Coll(\msf{A})$ and $\msf{C}\in \mathbf{Col}^r(\Sigma(\msf{A}, \msf{B}))$
\[\msf{cur}: \Sigma(\Sigma (\msf{A}, \msf{B}), \msf{C})\to \Sigma(\msf{A},\Sigma^\msf{A}(\mathsf{B},\mathsf{C}))\]
given by

\begin{equation}\label{equation: cur}
    \Lambda x. \{\mathbf{pair}\}(\{\mathbf{p}_1\}(\{\mathbf{p}_1\}(x)),x).
\end{equation}

\section{An internal language for $\Prop$ and $\Coll$}
{
As done in \cite{VOO08}, we consider the internal language of the functors $\Prop$ and $\Coll$. The former is in the form of a first-order typed language with equality, and the latter is in the form of one with dependent types.

A  context $\Gamma$ is an expression of the form $[x_0:\msf{A}_0, \dots, x_n:\msf{A}_n(x_1,\dots, x_{n-1})]$. A substitution of contexts $\bar{t}:\Gamma \to \Delta$, where $\Delta:=[y_0:\msf{B}_0, \dots, y_m:\msf{B}_m(y_0,\dots,y_{m-1})]$ is a list of terms $\bar{t}:=[t_0, \dots, t_m]$ such that $t_i:\msf{B}_i(t_0,\dots,t_{i-1})$ is in context $[\Gamma]$ for $0\le i\le m$.

We will adopt the notation $\msf{C}_x(y)$ for a type in context $x:\Aa, y:\msf{B}(x)$ and $\msf{f}_x:\msf{B}(x)\to \msf{B}'(x)$ for a term of type $\msf{B}'(x)$ in context $x:\msf{A}, y:\msf{B}(x)$. 


Recall from \cite{maietti_maschio_2021}, that contexts $\Gamma$ are interpreted as objects $\|\Gamma\|$ of the category $\CC_r$ and substitutions of contexts as arrows of $\CC_r$. A formula $\psi[\Gamma]$ in context $\Gamma$ is interpreted as an element of $\Prop(\|\Gamma\|)$, while a dependent type $\msf{A}[\Gamma]$ in context $\Gamma$ is interpreted as an object of $\Coll(\|\Gamma\|)$. Given the interpretations $\|\bar{t}\|:\|\Gamma\|\to \|\Delta\|$ and $\|\psi[\Delta]\|\in \Prop(\|\Delta\|)$, or $\|\msf{A}[\Delta]\|\in \Coll(\|\Delta\|)$, then the interpretation of $\psi(\bar{t})[\Gamma]$, or  $\Aa(\bar{t})[\Gamma]$, is respectively given by $\Prop_{\|t\|}(\|\psi\|)$, or $\Coll_{\|t\|}(\|\Aa\|)$.

Furthermore, we assume that given an interpretation of $\msf{A}[\Gamma]$ as an object $\|\Aa\|\in\Coll(\|\Gamma\|)$, then the interpretation of the context $[\Gamma, x:\Aa]$ is given by the object $\Sigma(\|\Gamma\|,\|\Aa\|)$ of $\CC_r$. In particular, given a substitution $\bar{t}:\Gamma\to \Delta$ and a type in context $\msf{A}[\Delta]$, then the substitution $[\bar{t},y]:[\Gamma,y:\Aa(\bar{t})]\to [\Delta, z:\Aa]$ is interpreted as the arrow 
\[\Sigma(\|\bar{t}\|, \|\Aa\|):\Sigma(\|\Gamma\|, \Coll_{\|\bar{t}\|}(\|\Aa\|))\to \Sigma(\|\Delta\|,\|\Aa\|)\]
If $\Aa[\Gamma]$ and $\msf{B}[\Gamma]$ are types in the same context and $f:\Aa\to \msf{B}$ is a term in context $[\Gamma, x:\Aa]$, then $f$ is interpreted as an arrow $\|f\|:\|\Aa\|\to \|\msf{B}\|$ in $\Coll(\|\Gamma\|)$ and the substitution $[id_\Gamma,f]:[\Gamma,x:\Aa]\to [\Gamma,y:\msf{B}]$ is interpreted as the arrow
\[\II_{\|\Gamma\|}(\|f\|):\Sigma(\|\Gamma\|,\|\Aa\|)\to \Sigma(\|\Gamma\|,\|\msf{B}\|)\]

If $\psi[\Gamma]$ is a formula in context, we define its validity $\Prop\vdash\psi[\Gamma]$ if and only if $\top_{\|\Gamma\|}\le_{\|\Gamma\|}\|\psi[\Gamma]\|$. Similarly, for a type $\Aa[\Gamma]$ in context,  we define $\Coll\vdash \Aa[\Gamma] $ if and only if there exists an arrow $1_{\|\Gamma\|}\to \|\msf{A}[\Gamma]\|$.

Since, the interpretation of the internal language of $\Prop$ extends Kleene realizability interpretation $\Vdash$ of $\mathsf{HA}$ (see \cite[Lemma 4.2]{maietti_maschio_2021}), and the elements of $\Prop(\|\Gamma\|)$ are equivalence classes of those in $\Coll(\|\Gamma\|)$, we shall adopt the notation $t\Vdash \msf{A}$ for a term $t:\Aa$ in context $\Gamma$ of a representative $\Aa$ of some proposition in $\Prop(\|\Gamma\|)$, while we shall adopt the notation $t\varepsilon \msf{A}$ for a term $t:\Aa$ in context $\Gamma$ when we do not refer to any underlying proposition.

}

\section{The predicative effective topos $\peff$}\label{section: descent}

In \cite{maietti_maschio_2021}, the predicative effective topos $\peff$ is defined as the base category $\mathcal{Q}_{\overline{\mathbf{Prop}^r}}$ of the elementary quotient completion of the elementary doctrine $\overline{\mathbf{Prop}^r}$. It is the category whose objects are the pairs $(\mathsf{A},[\mathsf{R}])$ where $\mathsf{A}$ is an object of the category $\CC_{r}$ and $[\mathsf{R}]$ is an object of $\overline{\mathbf{Prop}^{r}}(\mathsf{A}\times \mathsf{A})$ such that 
\begin{enumerate}
\item $\top_\Aa\le_\Aa\overline{\mathbf{Prop}^{r}}_{\Delta_\Aa}([\RR])$;
\item $[\RR]\le_{\Aa\times \Aa}\overline{\mathbf{Prop}^{r}}_{\langle p_2,p_1\rangle}([\RR])$, where $p_1,p_2:\Aa\times \Aa\to \Aa$ are the projections;
\item $\overline{\mathbf{Prop}^{r}}_{\langle p_1,p_2\rangle}([\RR]) \wedge \overline{\mathbf{Prop}^{r}}_{\langle p_2,p_3\rangle}([\RR])\le_{\Aa\times \Aa\times \Aa}\overline{\mathbf{Prop}^{r}}_{\langle p_1,p_3\rangle}([\RR])$,  where $p_1,p_2,p_3:\Aa\times \Aa\times \Aa\to \Aa$ are the  projections;
\end{enumerate}
Equivalently, the above conditions state that $[\RR]$ is an equivalence relation for the \emph{internal logic} of the doctrine $\overline{\mathbf{Prop}^{r}}$, see \cite{maietti_maschio_2021}.

	An arrow from $(\mathsf{A},[\mathsf{R}])$ to $(\mathsf{B},[\mathsf{S}])$ is an equivalence class $[\mathsf{f}]$ of arrows $\mathsf{f}:\mathsf{A}\rightarrow \mathsf{B}$ in $\CC_{r}$ such that $[\RR]\le_{\Aa\times \Aa}\overline{\mathbf{Prop}^{r}}_{\msf{f}\times \msf{f}}([\mathsf{S}])$, where $\mathsf{f}\sim \mathsf{g}$ if and only if $\top_A\le_\Aa\overline{\mathbf{Prop}^{r}}_{\langle \mathsf{f},\mathsf{g}\rangle} ([\mathsf{S}])$.

  Moreover, we recall from \cite{maietti_maschio_2021} that $\peff$ is equivalent to the exact completion $(\CC_r)_{ex/lex}$ of the left exact category $\CC_r$. Indeed, this follows from theorem~\ref{weaksub} and the fact that the elementary quotient completion of a weak subobject doctrine on a lex category $C$ is equivalent to the exact on lex completion of $C$ (see \cite{qu12,elqu}).

The following appears as \cite[Theorem 5.5]{maietti_maschio_2021}.
\begin{thm}\label{thm: peff lcc pretops list}
    The category $\peff$ is a locally cartesian closed list arithmetic pretopos.
\end{thm}

\section{Families of collections and sets over $\mathbf{pEff}$}

In this section, we define suitable categories for every object $(\msf{A},[\msf{R}]) \in \peff$ that are apt to interpret the notions of \emph{families of collections} and \emph{sets} in $\peff$. Slice categories provide a natural account of the notion of \emph{families of collections} over an object $(\msf{A},[\msf{R}]) \in \peff$; however, families of sets require additional considerations to be defined. To do that,
 by analogy with \cite{m09}, we define for every representative of $[\msf{R}]$ a category of \emph{extensional dependent collections} $\Dep(\msf{A},{\msf{R}})$ and show 
that it is equivalent to the category of suitable internal groupoid actions $\mathpzc{Des}(p,{\msf{R}})$. Since the objects of $\Dep(\msf{A},{\msf{R}})$ are defined through the indexed categories $\Coll$ and $\Prop$, we can easily define a notion of \emph{extensional dependent sets} and equivalently obtain a subcategory of $\peff/(\msf{A},[\msf{R}])$, independently of the representative ${\msf{R}}$.

The intuitive idea behind the definition of a family of collections 
over $(\mathsf{A},\mathsf{R})$ is the following: 
\begin{itemize}
\item first we take a family of realized collections $\mathsf{B}(a)$ over $\mathsf{A}$  in $\CC_{r}$;
\item then, we take an equivalence relation $\mathsf{S}_{a}(b,b')$ on each $\msf{B}(a)$ (depending on $\mathsf{A}$);
\item furthermore, we take a family of transport maps $b\mapsto \sigma_{a,a'}(r,b)$ for $a,a'\,\varepsilon\, \mathsf{A}$, $b\,\varepsilon\, B(a)$ and $r\Vdash\mathsf{R}(a,a')$ such that $\sigma_{a,a'}(r,b)\,\varepsilon\, \BB(a')$;




\item Finally, we ask $\sigma$ to satisfy the following conditions, which makes it a sort of action:

\begin{enumerate}
\item if $S_{a}(b,b')$ and $r\Vdash\mathsf{R}(a,a')$, then $\msf{S}_{a'}(\sigma_{a,a'}(r,b),\sigma_{a,a'}(r,b'))$, that is $\sigma$ respects $\mathsf{S}$;

\item if $r\Vdash\mathsf{R}(a,a)$ and $b\,\varepsilon\, \mathsf{B}(a)$, then $\msf{S}_{a}(b,\sigma_{a,a}(r,b)))$, that is $\sigma$ is well-behaved with respect to identities;

\item if $r\Vdash\mathsf{R}(a,a')$, $r'\Vdash\mathsf{R}(a',a'')$, $r''\Vdash\mathsf{R}(a,a'')$  and $b\,\varepsilon\, \mathsf{B}(a)$, then $$\msf{S}_{a''}(\sigma_{a,a''}(r'',b),\sigma_{a',a''}(r',\sigma_{a,a'}(r,b))))$$ that is $\sigma$ is well-behaved with respect to composition.
\end{enumerate}
Observe that,  taking $a'=a''$ and 
a witness for the reflexivity of $\RR$, conditions 2.\ and 3.\ above
imply the following condition stating that $\sigma$, once we fix two related objects $a,a'$, does not depend on the particular realizer of $\mathsf{R}(a,a')$:
\begin{enumerate}
\item[4.] ${\msf{S}}_{a'}(\sigma_{a,a'}(r,b),\sigma_{a,a'}(r',b))$ for every $b\,\varepsilon\, \BB(a)$, $a,a'\varepsilon \Aa$, $r,r'\Vdash \RR(a,a')$;
\end{enumerate}

\end{itemize}

To formulate this in a precise way this idea we need first to introduce the notion of fibered ${\overline{\textbf{Prop}^r}}$-equivalence relation:

\begin{dfn}\normalfont
	Let $\msf{A}$ be an object of $\CC_r$ and let $\msf{B}$ be an object of $\textbf{Col}^r(\msf{A})$. A \emph{fibered ${\overline{\mathbf{Prop}}^r}$-equivalence relation on $\msf{B}$}, or simply \emph{a fibered equivalence relation on} $\msf{B}$, is an object $[\msf{S}]$ in ${\overline{\textbf{Prop}^r}}(\Sigma(\msf{A},\msf{B}\times \msf{B}))$ such that 
	\begin{enumerate}
		\item $\top_{\Sigma(\msf{A},\msf{B})}\le_{\Sigma(\msf{A},\msf{B})}{\overline{\textbf{Prop}^r}}_{\textbf{I}_\msf{A}(\Delta_\msf{B})}([\msf{S}])$
		\item $[\msf{S}]\le_{\Sigma(\msf{A},\msf{B}\times \msf{B})}{\overline{\textbf{Prop}^r}}_{\textbf{I}_\msf{A}(\langle p_2,p_1\rangle)}([\msf{S}])$
		\item ${\overline{\textbf{Prop}^r}}_{\textbf{I}_\msf{A}(\langle p_1,p_2\rangle)}([\msf{S}])\wedge {\overline{\textbf{Prop}^r}}_{\textbf{I}_\msf{A}(\langle p_2,p_3\rangle)}([\msf{S}])\le_{\Sigma(\msf{A},\msf{B}\times \msf{B} \times \msf{B})} {\overline{\textbf{Prop}^r}}_{\textbf{I}_\msf{A}(\langle p_1,p_3\rangle)}([\msf{S}])$.
	\end{enumerate}
\end{dfn}


\begin{dfn}\label{dfn: actions 2}\normalfont
Let $(\mathsf{\msf{A}},[\mathsf{R}])$ be an object of $\textbf{pEff}$.  An \emph{extensional dependent collection} on $(\mathsf{\msf{A}} , [\mathsf{R}])$, with respect to ${\RR}$,  is a triple $(\msf{B},[\msf{S}],\sigma)$ where $\msf{B}\in\textbf{Col}^r(\msf{A})$, $[\msf{S}]$ is a  dependent ${\overline{\textbf{Prop}^r}}$-equivalence relation on $\msf{B}$ and 
	$$\sigma: \textbf{Col}^r_{p_1}(\msf{B})\times\RR  \to \textbf{Col}^r_{p_2}(\msf{B}) $$ 
	is an arrow in $\textbf{Col}^r(\msf{A}\times \msf{A})$ satisfying
 \begin{enumerate}
     \item\label{item: pseudo action 1} $\sigma$ preserves $[\msf{S}]$ i.e.\ for the arrow
\[\begin{tikzcd}[column sep= large]
	{\Sigma(\msf{A}\times \msf{A},\textbf{Col}^r_{p_1}(\msf{B}\times\msf{B}) \times \RR)} &&& {\Sigma(\msf{A}\times \msf{A},\textbf{Col}^r_{p_2}(\msf{B}\times\msf{B})) } 
		\arrow["{\textbf{I}((\sigma\circ \langle p_1, p_3\rangle)\times (\sigma\circ\langle p_2,p_3\rangle))}", from=1-1, to=1-4]
\end{tikzcd}\]
it holds
$$\overline{\textbf{Prop}^r}_{{\Sigma(p_1,\msf{B}\times\msf{B})\circ\textbf{I}(p_1)}}[\msf{S}]\le \overline{\textbf{Prop}^r}_{\Sigma(p_2,\msf{B}\times\msf{B})\circ {\textbf{I}((\sigma\circ \langle p_1, p_3\rangle)\times (\sigma\circ\langle p_2,p_3\rangle))}}[\msf{S}]$$ where $\Sigma(p_i,\BB\times\BB):\Sigma(\msf{A}\times \msf{A},\textbf{Col}^r_{p_1}(\msf{B}\times\msf{B}) )\to {\Sigma(\msf {A},\msf{B}\times\msf{B})}$, for $i=1,2$.

\item\label{item: pseudo action 3} $\sigma$ preserves the identities,   i.e. for the unique arrow $\rho$ making the following diagram commute
\[\begin{tikzcd}
	&& {\Sigma(\msf{A},\msf{B})} \\
	{\Sigma(\msf{A},\msf{B} \times \textbf{Col}^r_\Delta(\RR))} && {\Sigma(\msf{A},\msf{B}\times \msf{B})} \\
	&& {\Sigma(\msf{A},\msf{B})}
	\arrow["{\textbf{I}(p_1)}", from=2-1, to=1-3]
	\arrow["{\textbf{I}(\textbf{Col}^r_\Delta\sigma)}"', from=2-1, to=3-3]
	\arrow["{\textbf{I}(p_1)}"', from=2-3, to=1-3]
	\arrow["{\textbf{I}(p_2)}", from=2-3, to=3-3]
	\arrow["\rho", dashed, from=2-1, to=2-3]
\end{tikzcd}\]
then $\top\le {\overline{\textbf{Prop}^r}}_{\rho}([\msf{S}])$.

    \item  $\sigma$ preserves the composition, i.e. for the unique arrow $\tau$ \footnote{Here we use the notation $\msf{A}^3:=\msf{A}\times \msf{A}\times \msf{A}$}
    \[\begin{tikzcd}
        {\Sigma(\msf{A}^3, \textbf{Col}^r_{ p_1}(\msf{B}) \times \textbf{Col}^r_{\langle p_1,p_2 \rangle}(\RR) \times \textbf{Col}^r_{\langle p_2,p_3 \rangle}(\RR) \times \textbf{Col}^r_{\langle p_1,p_3 \rangle}(\RR))} \arrow[r,"{\tau}"]& \Sigma(\msf{A}^3,\textbf{Col}^r_{ p_3}(\msf{B}\times\BB)) 
    \end{tikzcd}\]
    such that
    \begin{itemize}
    \item$\II(p_2)\circ \tau= \II(\Coll_{\langle p_1,p_3\rangle}(\sigma)\circ \langle p_1,p_4\rangle)$  \item $\II(p_1)\circ\tau=\II(p_1)\circ\II(\Coll_{\langle p_2,p_3\rangle}(\sigma)\circ(\Coll_{\langle p_1,p_2\rangle}(\sigma)\times id)\times id)$
    \end{itemize}
     it holds that $\top\le {\overline{\textbf{Prop}^r}}_{\Sigma(p_3, \msf{B}\times \msf{B})\circ\tau}([\msf{S}])$.
    \end{enumerate}
   
\end{dfn}

\begin{dfn}\label{dfn: morphism actions 2}\normalfont
	Let $(\msf{A},[\msf{R}])$ be an object in $\peff$. A {morphism between extensional dependent collections} over  $(\mathsf{A},[\mathsf{R}])$ w.r.t. ${\RR}$ from $(\msf{B},[\msf{S}],\sigma)$ to $(\msf{C},[\msf{H}],\nu)$ 

 

 

 is an equivalence class ${[\msf{f}]}_{\cong} $ of arrows $\msf{f}:\msf{B}\to \msf{C}$ in $\textbf{Col}^r(\msf{A})$ such that 
	\begin{enumerate}
		\item\label{item: morph action 1}
  $\msf{f}$ respects the equivalence relations:
  $$[\msf{S}]\le\overline{\textbf{Prop}^r}_{\textbf{I}(\msf{f}\times \msf{f})}([\msf{H}])$$
		\item\label{item: morph action 2} 
  $\msf{f}$ commutes with the transport maps
 $\msf{H}(\nu_{a,a'}(r,\msf{f}(b)), \msf{f}(\sigma_{a,a'}(r,b)))$:
  
  $$\top\le \overline{\textbf{Prop}^r}_{\Sigma(p_2,\msf{C}\times\msf{C})\circ\mathsf{squ}}([\msf{H}])$$ where 
  $\mathsf{squ}:= \textbf{I}(\langle \nu\circ(\textbf{Col}^r_{p_1}(\msf{f}) \times 1_{\RR}) ,  \textbf{Col}^r_{p_2}(\msf{f})\circ\sigma\rangle)$; 
  
	\end{enumerate}
	and $\msf{f}\sim \msf{g}$ if and only if $[\msf{S}]\le\overline{\textbf{Prop}^r}_{\textbf{I}(\msf{f}\times \msf{g})}([\msf{H}])$. 
 
\end{dfn}

\begin{rmk}
The above notion was inspired by that of extensional dependent collection in the context of dependent type theory in
\cite{m09}.
\end{rmk}

\begin{dfn}
    Let $(\mathsf{\msf{A}},[\mathsf{R}])$ be an object of $\textbf{pEff}$.  An \emph{extensional dependent set} on $(\mathsf{\msf{A}} , [\mathsf{R}])$ w.r.t ${\RR}$  is an extensional dependent collection $(\msf{B},[\msf{S}],\sigma)$  on $(\mathsf{\msf{A}} , [\mathsf{R}])$ w.r.t ${\RR}$ such that $\msf{B}\in\textbf{Set}^r(\msf{A})$ and $[\msf{S}]\in {\overline{\textbf{Prop}_s^r}}(\Sigma(\msf{A},\msf{B}\times \msf{B}))$.
\end{dfn}


\begin{dfn}\normalfont
    \label{colset}

Given $(\msf{A},[\msf{R}])\in\peff$, we denote with $\Dep(\msf{A},{\msf{R}})$ the category of extensional dependent collections  on $(\mathsf{\msf{A}} , [\mathsf{R}])$ w.r.t ${\RR}$ and arrows between them.
Furthermore, we denote with  $\Depset(\msf{A},{\msf{R}})$ its full subcategory on \emph{extensional dependent sets}.
\end{dfn}

\begin{lemma}\label{lemma: indipendence representative}
     Let $(\msf{A},[\msf{R}])$ be an object in $\peff$ and let $\RR,\RR'$ be two representatives of $[\RR]$. The categories $\Dep(\msf{A},{\msf{R}})$ and $\Dep(\msf{A},{\msf{R}}')$ are equivalent. The same holds for extensional dependent sets.
     \begin{proof}
         The fact that $\RR$ and $\RR'$ are in $[\RR]\in \overline{\Prop}(\Aa\times \Aa)$ imply that there exist two arrows $h:\RR\to \RR'$ and $k:\RR'\to\RR$ in $\Coll(\Aa\times \Aa)$. We can construct the functor
         \[-\circ k:\Dep(\msf{A}, {\msf{R}})\to \Dep(\msf{A}, {\msf{R}}')\]
         that sends an extensional dependent collection $(\BB,[\Ss],\sigma)$ to the triple $(\BB,[\Ss],\sigma\circ (id \times k))$ and every arrow into itself. It is straightforward to prove this functor is an equivalence considering the similar functor  
         \[-\circ h:\Dep(\msf{A},  {\msf{R}}')\to \Dep(\msf{A}, {\msf{R}}).\]
         
     \end{proof}
\end{lemma}

In the following proposition, we show directly that extensional dependent collections induce objects in the slice categories of $\peff$.

\begin{prop}\label{prop: K equiv}
     Let $(\msf{A},[\msf{R}])$ be an object in $\peff$. There exists a full and faithful essentially surjective functor 
    $$\mathbf{K}:\Dep(\msf{A}, {\msf{R}})\to \peff/(\Aa,{[\msf{R}]}).$$   
\end{prop}
\begin{proof}
For every object $(\msf{B},\msf{S},\sigma)$ we define
\begin{equation}\label{defK}\mathbf{K}(\msf{B},[\msf{S}],\sigma):= [\msf{p}_1^\Sigma]_{\cong} : (\Sigma(\msf{A},\msf{B}),\exists_{d_{\bmR}}(\overline{\mathbf{Prop}^r}_{t_\sigma}([\mathsf{S}]))\to (\msf{A},[\msf{R}])
\end{equation}
for the commutative diagram
\[\begin{tikzcd}
	{\Sigma(\msf{A} \times \msf{A}, \RR \times \mathbf{Col}^r_{p_1}(\msf{B})\times \mathbf{Col}^r_{p_2}(\msf{B})) } && {\Sigma(\msf{A} \times \msf{A}, \RR \times \mathbf{Col}^r_{p_1}(\msf{B})\times \mathbf{Col}^r_{p_2}(\msf{B})) } \\
	{\Sigma(\msf{A} \times \msf{A}, \mathbf{Col }^r_{p_2}(\msf{B}) \times \mathbf{Col }^r_{p_2}(\msf{B})) } && {\Sigma(\msf{A} \times \msf{A}, \mathbf{Col }^r_{p_1}(\msf{B}) \times \mathbf{Col }^r_{p_2}(\msf{B})) } \\
	{\Sigma(\msf{A}, \msf{B}\times \msf{B})} && {\Sigma(\msf{A}, \msf{B}) \times \Sigma(\msf{A}, \msf{B})}
	\arrow["{\mathbf{I}(\langle \sigma\circ\langle p_1, p_2\rangle,p_3 \rangle)}"{description}, from=1-1, to=2-1]
	\arrow["{\Sigma(p_2, \msf{B}\times \msf{B})}"{description}, from=2-1, to=3-1]
	\arrow["{k}"{description}, from=2-3, to=3-3]
	\arrow[shift right= 18, "{t_\sigma}"', bend right=60pt, dashed, from=1-1, to=3-1]
	\arrow["{\mathbf{I}(\langle p_2, p_3\rangle)}"{description}, from=1-3, to=2-3]
	\arrow[shift left= 18,dashed,"{d_{\bmR}}", bend left=60pt, from=1-3, to=3-3]
\end{tikzcd}\]

Where $k$ is the isomorphism given by $\langle \Sigma(p_1, \msf{B})\circ \mathbf{I}(p_1),\Sigma(p_2, \msf{B})\circ \mathbf{I}(p_2)\rangle $.


Informally, the domain of the arrow above has as elements those pairs $(a,b)$ with $a\in \msf{A}$ and $b\in \mathsf{B}(a)$. A realizer for an equivalence of two such pairs $(a,b)$ and $(a',b')$ is a pair $(r,s)$ such that $r\Vdash \msf{R}(a,a')$ and $s\Vdash \msf{S}_{a'}(\sigma_{a,a'}(r,b),b')$.

For $[\msf{f}]:(\msf{B},[\msf{S}],\sigma)\rightarrow (\msf{C},[\msf{H}],\eta)$ we define $\mathbf{K}([\msf{f}])$ as $[\mathbf{I}(\msf{f})]$. 
These assignments will define a functor.

Since every arrow $h$ in $\CC_r$ from $\Sigma(\msf{A},\mathsf{B})$ to $\Sigma(\msf{A},\mathsf{C})$ such that $\msf{p}_1^{\Sigma}\circ h=\msf{p}_1^{\Sigma}$ 
has the form $\mathbf{I}(\widetilde{h})$ for some $\widetilde{h}:\msf{B}\rightarrow \msf{C}$ in $\mathbf{Col}^r(\msf{A})$, one obtains that every arrow $h$ in $\mathbf{pEff}/(\mathsf{A},[\msf{R}])$ between objects in the image of $\mathbf{K}$ has the form $[\mathbf{I}(\widetilde{h})]$. The fullness of $\mathbf{K}$ follows observing that such an arrow $\widetilde{h}$ determines in fact an arrow in $\Dep(\msf{A},\RR)$. This can be easily obtained using the properties of the original arrow $h$.

Faithfulness is immediate once one notice that equivalence of arrows $h$'s in $\mathbf{pEff}/(\mathsf{A},[\mathsf{R}])$ in the image of $\mathbf{K}$ corresponds exactly to equivalence of the corresponding $\widetilde{h}$'s in $\Dep(\msf{A},\RR)$.

 Now given an object $[\msf{f}]_{\cong}:(\msf{A}', [\msf{R}']) \to (\msf{A},[\msf{R}])$ in $\mathbf{pEff}/(\mathsf{A},\mathsf{R})$, we  build a triple $(\msf{B}^\msf{f}, [\msf{S}]^\msf{f}, \sigma^\msf{f})$ 
 in $\Dep(\msf{A},\bmR)$. In order to do that, we consider 
the arrow $p_1:\msf{A}\times \msf{A}'\to \msf{A}$ and  define $\msf{B}^\msf{f}:=\Sigma_{p_1}(\mathbf{Col}^r_{ (\msf{id}_{\msf{A}}\times \msf{f})}(\RR))$.\ 
Intuitively, $\msf{B}^\msf{f}(a)$ is given by pairs $(a', k)$ where $a'\in \msf{A}'$ and $k \Vdash \msf{R}(a, \msf{f}(a')) $.  
 Then we define 

$$[\msf{S}]^\msf{f}:=\overline{\mathbf{Prop}^r}_{(\mathsf{snd}^\msf{f}\times\mathsf{snd}^\msf{f})\circ\langle\II_\Aa(p_1),\II_\Aa(p_2)\rangle}([\msf{R}'])$$
where $$\mathsf{snd}^\msf{f}:= [\Lambda x.\{\mathbf{p}_2\}(\{\mathbf{p}_1\}(\{\mathbf{p}_2\}(x)))]_\approx:\Sigma(\Aa,\Sigma_{p_1}(\mathbf{Col}^r_{ (\msf{id}_{\msf{A}}\times \msf{f})}(\RR)))\to \Aa'$$ and $\II_\Aa(p_i):\Sigma(\Aa,\Sigma_{p_1}(\mathbf{Col}^r_{ (\msf{id}_{\msf{A}}\times \msf{f})}(\RR)\times \mathbf{Col}^r_{ (\msf{id}_{\msf{A}}\times \msf{f})}(\RR)))\to\Sigma(\Aa,\Sigma_{p_1}(\mathbf{Col}^r_{ (\msf{id}_{\msf{A}}\times \msf{f})}(\RR)))$, for $i=1,2$,
 to be the relation $\msf{R}'$ on the $\msf{A}'$ components of $\msf{B}^\msf{f}$.

 Finally, we can define $\sigma^\msf{f}$ as the arrow which sends a triple $((a', k),r)$ with $(a', k)\in \msf{B}^\msf{f}(a_1)$ and $r \Vdash \msf{R}(a_1, a_2)$ to the pair $(a', \msf{tr}(\msf{sym}(r), k))\in\msf{B}^\msf{f}(a_2)$ where $\msf{tr}(\msf{sym}(r), k)\Vdash\msf{R}(a_2, \msf{f}(a'))$ and $\msf{refl, sym}$ and $\msf{tr}$ are some chosen witnesses of the reflexivity, symmetry and transitivity of $\RR$. We leave this to the reader.

 Now we can observe that $\mathbf{K}(\msf{B}^\msf{f}, [\msf{S}]^\msf{f}, \sigma^\msf{f})\cong [\msf{f}]_{\cong}$. Indeed, there is an obvious arrow from $\Sigma(\msf{A},\msf{B}^\msf{f})$ to $\msf{A}'$ which takes the $\msf{A}'$ component of $\msf{B}$ to itself. Vice versa, given an element $a'\in\msf{A}'$, we can consider the triple $(\msf{f}(a'), (a', \msf{refl}))\in\Sigma(\msf{A},\msf{B})$. These arrows are inverse thanks to the relations $[\msf{R}']$ and $\exists_{d_{\bmR}}(\overline{\mathbf{Prop}^r}_{t_\sigma}([\mathsf{S}]^\msf{f}))$.
\end{proof}

\begin{rmk}\label{rmk: AC for the equivalence}\normalfont
 The above result provides an equivalence between the categories $\Dep(\msf{A},\RR)$ and $ \mathbf{pEff} /(\msf{A},[\msf{R}])$ in case we assume the axiom of countable choice in the meta-theory.
 
   Otherwise, we only obtain that $\mathbf{K}:\Dep(\mathsf{A},\RR)\rightarrow\mathbf{pEff}/(\mathsf{A},[\mathsf{R}])$ preserves and reflects all limits, colimits and exponentials. 
\end{rmk}

\begin{rmk}\label{rmk: equivalent categories same images}\normalfont
If $(\msf{A},[\msf{R}])$ is an object in $\peff$ and $\RR,\RR'$ are two representatives of $[\RR]$ as in Lemma \ref{lemma: indipendence representative}, then the image through $\mathbf{K}$ of the categories $\Dep(\msf{A}, {\msf{R}})$ and $\Dep(\msf{A}, {\msf{R}}')$ is the same, i.e.\ the following diagram commutes
    \[
\begin{tikzcd}
	{\Dep(\msf{A},{\msf{R}})} && {\Dep(\msf{A}, {\msf{R}}')} \\
	& {\peff/(\Aa,[\RR])}
	\arrow[""{name=0, anchor=center, inner sep=0}, "{-\circ k}", shift left=2, from=1-1, to=1-3]
	\arrow["\mathbf{K}"', from=1-1, to=2-2]
	\arrow[""{name=1, anchor=center, inner sep=0}, "{-\circ h}", shift left=2, from=1-3, to=1-1]
	\arrow["\mathbf{K}", from=1-3, to=2-2]
	\arrow["\cong", draw=none, from=0, to=1]
\end{tikzcd}\]
\end{rmk}

We conclude this section providing a direct correspondence between extensional dependent collections and internal \emph{groupoid action} (see \cite{maclane2012sheaves,facetI} for a precise definition) in $\peff$.
Indeed, for every object $(\msf{A},[\msf{R}])\in\peff$, the kernel pair of a quotient arrow 

\[p:= [id_{\Aa}]:(\msf{A},[=_\msf{A}])\to(\msf{A},[\msf{R}])\] 
induces an internal groupoid $\mathsf{Ker}(p, {\msf{R}})$ 
obtained through the kernel pair of $p$


\[
\begin{tikzcd}
	\cdots && {(\Sigma(\mathsf{A}\times\mathsf{A},{\mathsf{R}}),[=_{\mathsf{A}\times\mathsf{A}}])} && {(\mathsf{A},[=_\mathsf{A}])}
	\arrow[shift left=2, from=1-1, to=1-3]
	\arrow[shift right=2, from=1-1, to=1-3]
	\arrow[from=1-1, to=1-3]
	\arrow["{[\pi_1]}", shift left=2, from=1-3, to=1-5]
	\arrow["{[\pi_2]}"', shift right=2, from=1-3, to=1-5]
	\arrow[from=1-5, to=1-3]
\end{tikzcd}\]
and for any other representative ${\msf{R}}'\in[\msf{R}]$ the two internal groupoids are isomorphic in the category of internal groupoids of $\peff$.

Indeed, denoting with $\mathpzc{Des}(p, {\msf{R}})$ the category of the actions of $\mathsf{Ker}(p, {\msf{R}})$ in $\peff$, we obtain the following result.

\begin{prop}\label{prop: edp induce actions}
    Let $(\msf{A},[\msf{R}])$ be an object in $\peff$. There exists a full and faithful essentially surjective functor 
    $$A:\Dep(\msf{A}, {\msf{R}})\to \mathpzc{Des}(p, {\msf{R}}).$$
    \begin{proof}

By general descent theory, there exists a canonical comparison functor (see \cite[\S 2.2]{facetI})
\[\Phi^{p,{\msf{R}}}: \peff/(\msf{A},[\msf{R}])\to \mathpzc{Des}(p, {\msf{R}})\]
Since $p$ is a regular epimorphism and $\peff$ is exact, then we obtain that $p$ is an effective descent morphism, i.e.\ the functor $\Phi^{p, {\msf{R}}}$ is full, faithful and essentially surjective (see \cite[\S 2.4
]{facetI}). The functor $A$ is defined as the composition of $\mathbf{K}$ and $\Phi^{p,\RR}$.
    \end{proof}
\end{prop}


\section{The structure of fibred extensional sets  on $\peff$}
Now we are going to investigate the categorical structure  of  {\em extensional dependent sets}  in definition~\ref{colset}
over $\peff$  as a subfibration of the codomain fibration of $\peff$ defined as 
the Grothendieck construction  of the pseudo-functor
\[\peff_{set}:\peff\op\to\mathbf{Cat}\]

We are forced to deal with fibrations since  families of collections and sets  loose their functoriality under substitution when lifted over $\peff$ to become ``extensional".
At this point, given that 
 slice categories provide a natural account for the notion of \emph{families of collections} over an object of $\peff$ we will just work with subfibrations of the codomain fibration  over $\peff$.

For this purpose,
we exploit the functor $\mathbf{K}$ of Proposition \ref{prop: K equiv}. In more detail,
 the action of $\peff_{set}$ on an object $(\Aa,[\RR])$ is given by the full subcategory of $\peff/(\Aa,[\RR])$ whose objects are arrows isomorphic to arrows in $\mathbf{K}(\Depset(\msf{A}, {\msf{R}}))$ for (any) representative of $[\RR]$, see Remark \ref{rmk: equivalent categories same images}. If $[p]:(\Aa',[\RR'])\to(\ARp)$ is an arrow  in $\peff$, then the value of $\Set_{[p]}([f])$, for $[f]\in\peff_{set}((\Aa,[\RR]))$ is given by the equivalence class of the pullback of $[f]$ along $[p]$ that will be denoted by $[p]^*[f]$.
 
 To prove that the above pseudo-functor is well-defined we need the following lemma.

\begin{lemma}\label{lemma: well defined pseudo fuctor sets}
    Let $[p]:(\Aa',[\RR'])\to(\msf{A},[\msf{R}]),$ be an arrow in $\peff$. If $[f]\in\peff/(\ARp)$ is isomorphic to the image trough $\mathbf{K}$ of an extensional dependent set $(\BB,[\Ss],\sigma)\in \Depset(\msf{A}, {\msf{R}}) $, then $[p]^*[f]$ is isomorphic to the image through $\mathbf{K}$ of an element in $\Depset(\msf{A}, {\msf{R}'}) $.
    \begin{proof}
We assume that $[f]$ is isomorphic to $\mathbf{K}((\BB,[\Ss],\sigma))$, then we can choose two representative $f:\Aa'\to\Aa$  in $\CC_r$ and $[g]:\RR'\to\RR$ in $\overline{\Props}(\Aa\times \Aa)$ (recall that  $\overline{\Props}(\Aa\times \Aa)$ is the poset reflection of
    $\Set(\Aa\times \Aa)$)
for $[f]$ and consider the extensional dependent set 

$$(\Set_{p}(\msf{B}),(\overline{\Props})_{\Sigma(p,\msf{B}\times \msf{B})}([\msf{S}]),\mathbf{Set}^r_{p\times p}(\sigma)\circ (\mathsf{id}\times g))$$
in $\Depset(\msf{A}, {\msf{R}'}) $ and observe that its image through $\mathbf{K}$ is isomorphic to the pullback $[p]^*[f]$.
    \end{proof}
\end{lemma}

\begin{rmk}\label{rmk: set representable arrows in peff_set}
    Similar to Lemma \ref{rmk: I representable arrow}, any arrow $[\msf{f}]$ in $\peff_{set}(\Aa,[\RR])$ between two arrows isomorphic to the image through $\mathbf{K}$ of two objects in $(\msf{B},\Ss,\sigma)$ and  $(\msf{C},\msf{H},\eta)$ in
 $\Depset(\Aa,\RR)$
    \[[\msf{f}]:\mathrm{dom}(\mathbf{K}(\msf{B},\Ss,\sigma)\to \mathrm{dom}(\mathbf{K}(\msf{C},\msf{H},\eta))\]
    is itself in $\peff_{set}(\mathrm{dom}(\mathbf{K}(\msf{C},\msf{H},\eta)))$.
\end{rmk}

Denoting with $\peff_{set}^\to$ the Grothendieck construction of $\peff_{set}$ we obtain the following commutative diagram of fibrations

\[\begin{tikzcd}
		\\
		{{\mathbf{pEff}^\to_{set}}} && {{\mathbf{pEff}}^\to} \\
		\\
		& {\mathbf{pEff}}
		\arrow["{\mathsf{cod}_{set}}"', from=2-1, to=4-2]
		\arrow["{\mathsf{cod}}", from=2-3, to=4-2]
  \arrow[hook, from=2-1, to=2-3]
	\end{tikzcd}\]

We now study the categorical property of these fibrations. To do that, instead of working with slice categories over $(\Aa,[\RR])\in\peff$ we exploit the equivalent categories of extensional dependent collections and sets and prove their properties. The following results do not depend on the chosen representative of $[\RR]$ thanks to Lemma \ref{lemma: indipendence representative} and Remark \ref{rmk: equivalent categories same images}.

Before moving to the main theorem of the section we provide a characterization of monomorphisms for families of collections, which will be instrumental to prove part of the theorem. It exploits a characterization that holds for suitable doctrines in general. 
{
Indeed, there exists a forgetful functor 
\[U:\Dep(\mathsf{A},\RR)\to \peff\]
which ``forgets" the $\sigma$-component in a faithful way. It is obtained composing the functor $\mathbf{K}$ with the domain  functor $\peff/(\Aa,[\RR])\to \peff$.

}





\begin{lemma}\label{lemma: mono characterization}
Let $(\mathsf{A},[\mathsf{R}])$ be an object of \emph{$\textbf{pEff}$}. An arrow $[\msf{m}]:(\mathsf{B},[\msf{S}],\sigma)\rightarrow (\mathsf{C},[\msf{H}],\nu)$ in $\Dep(\mathsf{A},\bmR)$ is a monomorphism if and only if
$$[\msf{S}]=\overline{\mathbf{Prop}^r}_{\II(\msf{m}\times \msf{m})}([\msf{H}])$$

\begin{proof}
{ Since $U$ is faithful, it reflects monomorphisms. Hence, using \cite[Corollary 4.8]{qu12} for the elementary quotient completion of ${(\overline{\Prop})}$ we obtain that the arrow $$U([\msf{m}])=[\II(\msf{m})]:(\Sigma(\Aa,\msf{B}),\exists_{d_\RR}\overline{\Prop}_{t_\sigma}[\Ss])\to (\Sigma(\Aa,\msf{C}),\exists_{d_\RR}\overline{\Prop}_{t_\nu}[\msf{H}])$$
is a monomorphism if and only if 
\[\exists_{d_\RR}\overline{\Prop}_{t_\sigma}[\Ss]=\overline{\Prop}_{\II(\msf{m})\times\II(\msf{m})}(\exists_{d_\RR}\overline{\Prop}_{t_\nu}[\msf{H}])\]
which is equivalent to 
$$[\msf{S}]=\overline{\mathbf{Prop}^r}_{\II(\msf{m}\times \msf{m})}([\msf{H}])$$
by definition of the arrows $d_\RR$ and $t_\sigma$ in the proof of Proposition \ref{prop: K equiv}.

}

\end{proof}
\end{lemma}

{ 
\begin{rmk}
    The same result can be obtained  in this alternative way. Consider to lift the doctrine
 $\Prop$ over  the slice of $\CC_r$ over $\Aa\in\CC_r$ as done in \cite{cioffo2023biased}: 
$$(\overline{\Prop})_{/\msf{A}}: (\CC_r/\Aa)\op\rightarrow\mathbf{Pos}$$
   
    the action on an object $\msf{f}:\msf{X}\to \Aa$ of $\CC_r$ is given by $\Prop(\msf{X})$, and the action on an arrow from $\msf{f}$ to $\msf{g}:\msf{Y}\to \Aa$ (which is an arrow $\msf{h}:\msf{X}\to\msf{Y}$ such that $\msf{g}\circ \msf{h}=\msf{f}$) is given by $\Prop_{\msf{f}}$. Then,  consider the
 elementary quotient completion of ${(\overline{\Prop})_{/\msf{A}}}$ and  observe that
there exists a forgetful functor 
\[U:\Dep(\mathsf{A},\RR)\to \mathcal{Q}_{(\overline{\Prop})_{/A}}\]
which ``forgets" the $\sigma$-component in a faithful way: it sends an object $(\msf{B},\Ss,\sigma)$ to the pair $(\msf{p}_1^\Sigma, \Ss)$ in $\mathcal{Q}_{(\overline{\Prop})_{/A}}$, where $\msf{p}_1^\Sigma:\Sigma(\Aa,\BB)\to \Aa$.
Since, as shown in \cite{cioffo2023biased}, the elementary quotient completion of ${(\overline{\Prop})_{/A}}$ inherits some properties from $\overline{\Prop}$, and in particular it has full strict comprehensions and comprehensive diagonals,  \cite[Corollary 4.8]{qu12} applied to it and faithfulness of $U$ imply the same monomorphism characterization of lemma~\ref{lemma: mono characterization}.

\end{rmk}

    }

\begin{thm}\label{thm: pretopos}
	Let $(\mathsf{A},[\mathsf{R}])$ be an object of \emph{$\textbf{pEff}$}. The categories \emph{$\Dep(\msf{A},\RR)$} and \emph{$\Depset(\msf{A},\RR)$} are locally cartesian closed list-arithmetic pretoposes.
    \begin{proof}
    Since, $\mathbf{pEff}$ is equivalent to an exact completion of a lex category (see \cite{maietti_maschio_2021}) and \break$\mathbf{K}:\Dep(\msf{A}, {\msf{R}})\rightarrow\mathbf{pEff}/(\mathsf{A},[\mathsf{R}])$ preserves and reflects all limits, colimits and exponentials, in the case of families of collection Theorem \ref{thm: pretopos} follows from the properties of the exact completion. Indeed, exactness and local cartesian closure are preserved under slicing, such as extensivity (see \cite{carboniextensive}) and parametrized list objects (see \cite{maiettimodular}).    
    However, when restricting to $\mathbf{pEff}_{set}(\msf{A},[\msf{R}])$ we must check that the above properties restrict to a proper subcategory of $\mathbf{pEff}/(\mathsf{A},[\mathsf{R}])$. Hence, we provide directly a proof of the result for \emph{$\Dep(\msf{A},\RR)$} to be able to restrict to \emph{$\Depset(\msf{A},\RR)$} and hence to $\mathbf{pEff}_{set}(\msf{A},[\msf{R}])$.\\


{\emph{Finite limits.}}

\begin{enumerate}
    \item  A terminal object in $\Dep(\msf{A},\RR)$ is given by $(1,\top,!)$. This is a terminal object also in $\Depset(\msf{A},\RR)$.
    \item A binary product of $(\msf{B},[\msf{S}],\sigma) $ and $(\msf{C},[\msf{H}], \nu)$ is given by
    \[(\msf{B}\times \msf{C}, \overline{\textbf{Prop}^r}_{\textbf{I}( p_1 \times p_1)}([\msf{S}]) \wedge \overline{\textbf{Prop}^r}_{\textbf{I}( p_2 \times p_2)}([\msf{H}]),\langle \sigma \circ (p_1 \times \msf{id}_\msf{R}), \nu \circ(  p_2 \times \msf{id}_\msf{R})\rangle) \]
    with projections inherited by those of $\msf{B}\times \msf{C}$. Since the product of two sets is a set and the conjunction of small propositions is a small proposition, the above construction restricts to $\Depset(\msf{A},\RR)$.
    \item An equalizer of two arrows $[\msf{f}],[\msf{g}]: (\msf{B},[\msf{S}],\sigma)\rightarrow\CUpsi$ is given by 
    \[(\Sigma^\msf{A}(\msf{B}, \mathbf{Col}^r_{\mathbf{I}(\langle {\msf{f}},\msf{g} \rangle) }\msf{H}), \overline{\textbf{Prop}^r}_{\textbf{I}(\msf{p}_1^{\msf{A},\Sigma}\times \msf{p}_1^{\msf{A},\Sigma})}([\msf{S}]),\eta)\]
    { where $\eta$ acts on a realizer $r\Vdash \RR(a,a')$, a term $b:\msf{B}(a)$
    and a realizer $q:\msf{H}_a(\msf{f}(a),\msf{g}(a))$, as $\sigma_{a,a'}(r,b)$ and $q':\msf{H}_{a'}(\msf{f}(\sigma_{a,a'}(r,b)),\msf{g}(\sigma_{a,a'}(r,b)))$, where $q'$ is obtained through the concatenation of the term $q_\nu:\msf{H}_{a'}(\nu_{a,a'}(r,\msf{f}(b)),\nu_{a,a'}(r,\msf{g}(b)))$ (from condition \ref{item: pseudo action 1} of Definition \ref{dfn: actions 2} for $\nu$) and the terms $u:\msf{H}_{a'}(\msf{f}(\sigma_{a,a'}(r,b)),\nu_{a,a'}(r,\msf{f}(b)))$ and $v:\msf{H}_{a'}(\nu_{a,a'}(r,\msf{g}(b)),\msf{g}(\sigma_{a,a'}(r,b)))$ (from condition \ref{item: morph action 2} of Definition \ref{dfn: morphism actions 2} for $\msf{f}$ and $\msf{g}$ and symmetry).
    }
\end{enumerate}


{\emph{Stable finite coproducts}.}
  
An initial object is given by $(0,\top, \sigma)$ where $\sigma$ is the unique possible one. A coproduct for $(\mathsf{B},[\msf{S}],\sigma)$ and $(\mathsf{C},[\msf{H}],\nu)$ is given by 
$$(\mathsf{B}+\mathsf{C},[\msf{S}]{+}[\msf{H}],(\sigma{+}\nu)j)$$

where 
\begin{enumerate}
\item $[\msf{S}]{+}[\msf{H}]=\exists_{\mathbf{I}(i_1\times i_1)}([\msf{S}])\vee \exists_{\mathbf{I}(i_2\times i_2)}([\msf{H}])$,
\item $j$ is the canonical isomorphism from $(\mathbf{Col}^r_{p_1}(\msf{B})+\mathbf{Col}^r_{p_1}(\msf{C}))\times \bmR$ to $ (\mathbf{Col}^r_{p_1}(\msf{B})\times \bmR)+(\mathbf{Col}^r_{p_1}(\msf{B})\times \bmR)$ which exists since $\mathbf{Col}^{r}(\msf{A}\times \msf{A})$ is extensive.

\end{enumerate}
together with the injections $[i_1]$ and $[i_2]$. One can check 
  that these injections are mono, and that coproducts are stable under pullbacks using pullbacks constructed through the finite limits of \emph{$\Dep(\msf{A},\RR)$} and \emph{$\Depset(\msf{A},\RR)$} described above. Initial objects are trivially stable, since they are stable in $\mathbf{Col}^{r}(\mathsf{A})$.

These constructions provide stable finite coproducts also in $\Depset(\msf{A},\RR)$ using the same arguments together with Theorem \ref{thm: small Prop^r properties} and Lemma \ref{rmk: I representable arrow}.\\
    
  {\emph{List objects.}} 
A list object for $(\msf{B},[\msf{S}],\sigma)$ is given by $(\mathsf{List}(\mathsf{B}),[\mathsf{S}_{list}],\sigma_{list})$
where
\begin{enumerate} 
\item { the relation $\Ss_{list}$ holds on two lists $x,y$ if they have the same length $l(x)=_\msf{N}l(y)$ and they have the same components: \ $\forall n\ \varepsilon \ \msf{N}(\exists u\ \varepsilon\ \msf{B}(a) 
\ \exists v\ \varepsilon\  \msf{B}(a)(\Ss_a(u,v)\wedge i_1(u)=_{\msf{B}(a)}\msf{comp}(x,n)\wedge i_1(v)=_{\msf{B}(a)}\msf{comp}(y,n))\vee \exists u'\varepsilon \ 1 \ \exists v'\varepsilon \ 1 \ (i_2(u')=_{1}\msf{comp}(x,n)\wedge i_2(v')=_{1}\msf{comp}(y,n))      )$, i.e.}
 $$\mathsf{S}_{list}=\overline{\Prop}_{\mathbf{I}(\ell\times \ell)}(\exists_{\mathbf{I}(\Delta_{\mathsf{N}})}(\top))\wedge\qquad\qquad\qquad\qquad\qquad\qquad\qquad $$ $$\forall_{\mathbf{I}(p_1)}(\overline{\Prop}_{\mathbf{I}((\langle\mathsf{comp}\circ (p_1\times \mathsf{id}_{\mathsf{N}}))\times (\langle\mathsf{comp}\circ (p_2\times \mathsf{id}_{\mathsf{N}})))}(\exists_{\mathbf{I}(i_1\times i_1)}([\mathsf{S}])\vee \exists_{\mathbf{I}(i_2\times i_2)}(\top) ))$$

where $\ell$ is the length arrow from $\mathsf{List}(\msf{B})\rightarrow \mathsf{N}$ in $\mathbf{Col}^r(\mathsf{A})$ and $\mathsf{comp}:\mathsf{List}(\msf{B})\times \mathsf{N}\rightarrow \msf{B}+1$ is the component arrow in $\mathbf{Col}^r(\mathsf{A})$ which selects the $n$-component of a list.

\item  $\sigma_{list}$ is the unique arrow in $\mathbf{Col}^r(\mathsf{A})$ making the following diagram commute

\[\begin{tikzcd}
	{\bmR} & {\msf{List}(\mathbf{Col}^r_{p_1}(\msf{B}))\times \RR} & {(\msf{List}(\mathbf{Col}^r_{p_1}(\msf{B}))\times \mathbf{Col}^r_{p_1}(\msf{B})) \times \RR} \\
	& {\msf{List}(\mathbf{Col}^r_{p_2}(\msf{B}))} & {\msf{List}(\mathbf{Col}^r_{p_2}(\msf{B}))\times \mathbf{Col}^r_{p_2}(\msf{B})}
	\arrow["{\langle\epsilon\circ!,\msf{id}_{\bmR}\rangle}", from=1-1, to=1-2]
	\arrow["{\epsilon \circ !}"', from=1-1, to=2-2]
	\arrow["{\sigma_{list}}", from=1-2, to=2-2]
	\arrow["{\msf{cons}}"', from=2-3, to=2-2]
	\arrow["{\msf{cons}\times \msf{id}_{\bmR}}"', from=1-3, to=1-2]
	\arrow["{\langle\sigma_{list}\circ(p_1\times \msf{id}_{\RR}), \sigma \circ (p_2 \times \msf{id}_{\RR})\rangle}", from=1-3, to=2-3]
\end{tikzcd}\]
\end{enumerate}
These constructions can be performed also in $\Depset(\mathsf{A},\RR)$ thanks to Theorem \ref{thm: small Prop^r properties} and Lemma \ref{rmk: I representable arrow}.\\

{\emph{Exactness.}}

Recall that $\peff$ is exact and that the slice categories of an exact category are also exact. Hence, exactness of $\Dep(\Aa,\RR)$ follows from the fact that  $\mathbf{K}:\Dep(\Aa,\RR)\to \peff/(\Aa,[\RR])$ is a fully and faithful essentially surjective functor. To prove that $\Depset(\Aa,\RR)$ is also exact, 
observe that the coequalizer of an equivalence relation on $(\BB,[\Ss], \sigma)\in\Depset(\Aa,\RR)$
$$[\msf{c}]:\CUpsi\to (\BB,[\Ss], \sigma)\times (\BB,[\Ss], \sigma) $$

 is given by
$$[\msf{id}_{\BB}]:(\BB,[\Ss], \sigma)\to(\BB, [\Ss'], \sigma)$$

where

$$[\Ss']:= \exists_{\II(\langle p_2,p_3\rangle)}(\overline{\Prop}_{\II(\msf{c_1}\times p_1)}([\Ss])\wedge \overline{\Prop}_{\II(\msf{c_2}\times p_2)}([\Ss]))$$
supposed $[\msf{c}_1]=p_1\circ [\msf{c}]$ and $[\msf{c}_2]=p_2\circ [\msf{c}]$.

The image through $\mathbf{K}$ of $[\msf{id}_\msf{B}]$ 
\[[\II(\msf{id}_\msf{B})]:(\Sigma(\Aa,\msf{B}),\exists_{d_\msf{R}}\overline{\Prop}_{t_\sigma}([\msf{S}]))\to (\Sigma(\Aa,\msf{B}),\exists_{d_\msf{R}}\overline{\Prop}_{t_\sigma}([\msf{S}']))\]
is a coequalizer in $\peff$ of $\mathbf{K}([\msf{c}])$, that is in $\peff_{set}(\Sigma(\Aa,\msf{B}),\exists_{d_\msf{R}}\overline{\Prop}_{t_\sigma}([\msf{S}']))$ by Remark \ref{rmk: set representable arrows in peff_set}. Stability and effectiveness follow from the pullback property in Lemma \ref{lemma: well defined pseudo fuctor sets}.
\\

\emph{Local cartesian closure.}
         For two objects $(\mathsf{B},[\msf{S}],\sigma)$ and $(\mathsf{C},[\msf{H}],\nu)$ in $\Dep(\msf{A},\RR)$ consider the triple $( \msf{C}^{\msf{B}}, [\msf{H}]^{[\msf{S}]}, 
         {\nu}^{\sigma})$ defined as follows.
The object $\msf{C}^{\msf{B}}:= \Sigma^{\msf{A}}(\msf{B}\Rightarrow \msf{C},\Pi_{p_1}({\Coll}_{\II(p_2)}(\msf{S})\to {\Coll}_{\II(\langle \msf{ev}\circ (\msf{id}\times p_1),\msf{ev}\circ (\msf{id}\times p_2) \rangle)}(\msf{H}))$ is obtained through the weak exponential $\msf{B}\Rightarrow \msf{C}$, $\msf{ev}$ in $\mathbf{Col}^r(\msf{A})$. Intuitively, we restrict to those arrows $f_a\in \msf{B}(a) \Rightarrow \msf{C}(a)$ whose evaluation preserves $\msf{S}$ and $\msf{H}$, i.e.\ such that we have $s\Vdash (\forall b,b'\in \msf{B}(a))\ (\msf{S}_a(b,b')\to \msf{H}_a(f_a(b),f_a(b')))$.

         The relation $[\msf{H}]^{[\msf{S}]}$ is given by $\forall_{\II({p_1})}({\overline{\Prop}}_{\II(\langle \msf{ev}\circ (p_1\times \msf{id}),\msf{ev}\circ (p_2\times\msf{id}) \rangle)}([\msf{H}]))$. Intuitively, we relate two arrows $f_a, g_a\in \msf{B}(a) \Rightarrow \msf{C}(a)$ when we have $r\Vdash (\forall b\in \msf{B}(a))\ \msf{H}_a(f_a(b),g_a(b))$. 

         The action $\sigma^{\nu}$ is defined on a triple $((f_{a_1}, s), r)$, where $r\Vdash \msf{R}(a_1, a_2)$, as the pair $(\Tilde{f}_{a_2}, \Tilde{s})$ with $\Tilde{f}_{a_2}(b):= \nu (f_{a_1}(\sigma(b, \msf{sym}(r))), r)$ and  $\Tilde{s}$
         is obtained through $s$ and compatibility of   $\sigma$ and $\nu$ with the corresponding equalities in condition (\ref{item: pseudo action 1}) of Definition \ref{dfn: actions 2}, and   a witness $\msf{sym}$ of the symmetry of $\RR$.

         We now observe that the evaluation arrow $\msf{ev}: (\msf{B}\Rightarrow\msf{C})\times \msf{B}\to \msf{C}$  preserves $[\msf{H}]^{[\msf{S}]}$, $[\msf{H}]$ and $[\msf{S}]$. In the sense that, if $((f_a,s),b)$ and $((f'_a, s'),b')$  are such that $r\Vdash \msf{H}_a(f_a(b),f'_a(b))$  and $t\Vdash \msf{S}(b, b')$, then using $s$ (applied to $t$), $r$ and the transitivity of $[\msf{H}]$ we obtain a $q\Vdash \msf{H}_a(f_a(b),f'_a(b'))$.

         Hence, the evaluation arrow of $( \msf{C}^{\msf{B}}, [\msf{H}]^{[\msf{S}]}, 
         {\nu}^{\sigma})$ is given by $[\msf{ev}]$ and we get an exponential object instead of a weak exponential as it happened in $\mathbf{Col}^r$. 

        Since weak exponentials of two sets are sets and small propositions are closed under connectives and quantification over sets, we obtain that the above construction restricts to $\Depset(\msf{A},\RR)$.

The local cartesian closure of $\Dep(\msf{A},\RR)$ follows observing that \break$\mathbf{K}:\Dep(\mathsf{A},\RR)\rightarrow\mathbf{pEff}/(\mathsf{A},[\mathsf{R}])$ induces a fully faithful and essentially surjective functor between the slice categories
        $$\mathbf{K}/(\msf{B},[ \msf{S}], \sigma):\Dep(\mathsf{A},\RR)/(\msf{B}, [\msf{S}], \sigma)\rightarrow(\mathbf{pEff}/(\mathsf{A},[\mathsf{R}]))/\mathbf{K}(\msf{B}, [\msf{S}], \sigma)$$
        for every object $(\msf{B}, [\msf{S}], \sigma)\in \Dep(\mathsf{A},\RR)$. Moreover, we have the isomorphism 
        $$(\mathbf{pEff}/(\mathsf{A},[\mathsf{R}]))/\mathbf{K}(\msf{B}, [\msf{S}], \sigma)\cong\mathbf{pEff}/\msf{dom}(\mathbf{K}(\msf{B}, [\msf{S}], \sigma))$$
        and the fact that $\peff$ is locally cartesian closed, implies that $\Dep(\mathsf{A},\RR)$ is locally cartesian closed by reflection of exponentials. To prove that also $\Depset(\mathsf{A},\RR)$ is locally cartesian closed, we observe that for every object $(\msf{B}, [\msf{S}], \sigma)\in \Depset(\mathsf{A},\RR)$, $\mathbf{K}$ induces a fully faithful and essentially surjective functor
        $$\mathbf{K}/(\msf{B},[ \msf{S}], \sigma):(\Depset(\Aa,\RR))/(\msf{B}, [\msf{S}], \sigma)\rightarrow(\mathbf{pEff}_{set}(\mathsf{A},[\mathsf{R}]))/\mathbf{K}(\msf{B}, [\msf{S}], \sigma)$$ and we also have the isomorphism 
$$(\mathbf{pEff}_{set}(\mathsf{A},[\mathsf{R}]))/\mathbf{K}(\msf{B}, [\msf{S}], \sigma)\cong\mathbf{pEff}_{set}(\msf{dom}(\mathbf{K}(\msf{B}, [\msf{S}], \sigma)))$$
by Remark \ref{rmk: set representable arrows in peff_set},
        and $\mathbf{K}$ induces a fully faithful essentially surjective functor
        $$\mathbf{K}:\Depset(\msf{dom}(\mathbf{K}(\msf{B}, [\msf{S}], \sigma))\to\mathbf{pEff}_{set}(\msf{dom}(\mathbf{K}(\msf{B}, [\msf{S}], \sigma)))$$
        Hence, by preservation and reflection of exponentials we obtain that $\Depset(\Aa,\RR)$ is locally cartesian closed.

\end{proof}
\end{thm}


\section{The "small subobjects" classifier}

Let $\peffpr:\peff^{op}\to\mathbf{Pos}$ denote the elementary quotient completion of $\overline{\Prop}:\CC_r\op\to\mathbf{Pos}$ whose value on an object $(\Aa,[\RR])$ is given by the descent data of $\overline{\Prop}(\Aa)$, i.e.

\begin{equation}\label{equation: descent}
    \peffpr(\Aa,[\RR]):=\left(\left\{[\msf{P}]\in\overline{\Prop}(\Aa)\ |\ \; \overline{\Prop}_{p_1}([\msf{P}])\wedge [\RR] \le \overline{\Prop}_{p_2}([\msf{P}])\right\},\le\right)\end{equation}
where $p_i:\Aa\times \Aa\to \Aa$ are the projections for $i=1,2$, and whose action on arrows is given by that of $\overline{\Prop}$.

Let $\peffprs:\peff\op\to\mathbf{Pos}$ denote the restriction of $\peffpr$ on the fibers of $\overline{\Props}:\CC_r\op\to \mathbf{Pos}$. From \cite{maietti_maschio_2021} we recall the following result which appears as Theorem 5.7.

\begin{thm}\label{thm: propq iso sub}
    The functor $\peffpr$ is a first-order hyperdoctrine equivalent to the subobjects doctrine $\peffpr\cong \mathbf{Sub}_{\peff}$.
    \qed
\end{thm}

We now provide a similar correspondence for $\peffprs$ regarding suitable subobjects. We will denote with $\mathbf{Sub}_{\peff}^\mathbf{s}$ the doctrine which sends an object $(\msf{A},[\msf{R}])\in\peff$ to the poset of subobjects of $\peff$ over $(\msf{A},[\msf{R}])$ which have a representative in $\peff_{set}^\to$. This is a well-defined functor thanks to Lemma \ref{lemma: well defined pseudo fuctor sets}. We will refer to these objects as \emph{small subobjects}.

\begin{lemma}\label{lemma: mono charac 2}
    Let $(\mathsf{A},\mathsf{R})$ be an object of \emph{$\textbf{pEff}$}. Given an element $(\BB, [\Ss],\sigma)\in \Dep(\Aa,[\RR])$, $\mathbf{K}(\BB, [\Ss],\sigma)$ is a monomorphism in $\peff$ if and only if $[\Ss]=\top$.
    \begin{proof}
       This follows from Lemma \ref{lemma: mono characterization}, and the fact that an arrow $m$ in $\peff$ with codomain $(\Aa,[\RR])$ is a monomorphism  if and only if it is a monomorphism in $\peff/(\Aa,[\RR])$ seen as an arrow from $m$ to $id_{(\Aa,[\RR])}$. Indeed, $id_{(\Aa,[\RR])}$ is isomorphic to $\mathbf{K}(1,\top, !)$.
    \end{proof}
    
\end{lemma}

\begin{prop}
    The functor $\peffprs$ is equivalent to the small subobjects doctrine $\mathbf{Sub}_{\peff}^\mathbf{s}$.
    \begin{proof}
        Given an object $(\Aa,[\RR])\in\peff$, consider a subobject over it which has a representative of the form $\mathbf{K}(\BB, [\Ss],\sigma)$ where $(\BB, [\Ss],\sigma)\in\Depset(\Aa,\RR)$. We can associate to it the element $[\BB]\in\overline{\Props}(\Aa)$. Thanks to $\sigma$, it follows that $[\BB]\in \peffprs(\Aa, [\RR])$. Moreover, for two monomorphisms of this form, it holds that $\mathbf{K}(\BB, [\Ss],\sigma)\le \mathbf{K}(\BB', [\Ss'],\sigma')$ if and only if $[\BB]\le[\BB']$ in $ \peffprs(\Aa, [\RR])$. Conversely, given a  $[\msf{P}]\in \peffprs(\Aa, [\RR])$ and a representative $\sigma: \Coll_{p_1}(\msf{P})\times \RR\to\Coll_{p_2}(\msf{P})$ (whose existence is guaranteed by (\ref{equation: descent})), we associate the equivalence class of $\mathbf{K}(\msf{P},\top,\sigma)$ which is a monomorphism thanks to Lemma \ref{lemma: mono charac 2}. This provides the desired equivalence.
    \end{proof}
\end{prop}

Applying the same argument in  \cite[Thm. 5.12]{maietti_maschio_2021}, and combining it with the previous proposition, we obtain a classifier for small subobjects.

\begin{thm}\label{thm: small sub classifier}
There exists an object $\Omega$ which represents $\mathbf{Sub}_{\peff}^\mathbf{s}$, i.e.\ there exists a natural isomorphism between the functors $\mathbf{Sub}_{\peff}^\mathbf{s}(-)$ and $\peff(-,\Omega)$.
    \qed
\end{thm}

 We now summarize the main results obtained in this work.
\begin{thm}\label{main theorem}
    The category $\peff$ is equipped with two fibrations 
    \[\begin{tikzcd}
		\\
		{{\mathbf{pEff}^\to_{set}}} && {{\mathbf{pEff}}^\to} \\
		\\
		& {\mathbf{pEff}}
		\arrow["{\mathsf{cod}_{set}}"', from=2-1, to=4-2]
		\arrow["{\mathsf{cod}}", from=2-3, to=4-2]
  \arrow[hook, from=2-1, to=2-3]
	\end{tikzcd}\]
satisfying the following properties:
 \begin{enumerate}
      \item $\peff$ is a locally cartesian closed list arithmetic pretopos;
     
     \item for each $(\Aa,[\RR])\in\peff$ the fiber $\mathsf{cod}_{set}^{-1}(\Aa,[\RR])$ 
     is a locally cartesian closed list arithmetic pretopos;
     \item for every arrow $[f]:(\Aa,[\RR])\to (\Aa',[\RR'])$ in $\peff$, the functor $$[f]^*:\mathsf{cod}_{set}^{-1}(\Aa',[\RR'])\to \mathsf{cod}_{set}^{-1}(\Aa,[\RR])$$ preserves the locally cartesian closed  list-arithmetics pretopos structure;
     \item for each $(\Aa,[\RR])\in\peff$ the inclusion $\mathsf{cod}_{set}^{-1}(\Aa,[\RR])\to \peff/(\Aa,[\RR])$ preserves the locally cartesian closed  list-arithmetics pretopos structure;
     \item there exists an object $\Omega$ classifying $\mathbf{Sub}_{\peff}^\mathbf{s}$;
     \item for every $[f]\in \mathsf{cod}_{set}^{-1}(\Aa,[\RR])$ there exists an exponential object $(\pi_\Omega)^{[f]}$ in $\peff/(\Aa,[\RR])$ where $\pi_{\Omega}:(\Aa,[\RR])\times \Omega \to (\Aa,[\RR])$ is the first projection;
     \item $\peff$ satisfies the \emph{formal Church's thesis} (see  \cite[Theorem 5.8]{maietti_maschio_2021}).
 \end{enumerate}
   \begin{proof}
       1.\ follows from Theorem \ref{thm: peff lcc pretops list}, 2.\ follows from Theorem \ref{thm: propq iso sub}, 3.\, 4.\, 5.\ and 6.\ follow from the constructions defined to prove Theorem \ref{thm: pretopos}. 7.\ follows from Theorem \ref{thm: small sub classifier}.
   \end{proof} 
\end{thm}
The above properties testify that $\peff$ is a 
{\it fibred predicative variant} of Hyland's Effective Topos. Indeed, the above properties are validated after substituting $\peff$ with $\eff$ and identified fibred sets with all the objects of a slice category of $\eff$.


\section{Towards a direct interpretation of $\mathbf{emTT}$ into $\peff$}
The realizability interpretation of the intensional level of the Minimalist Foundation in \cite{IMMSt,mmr21,mmr22} can be turned into a morphism of doctrines from the syntactical doctrine $G^{\mtt}$ arising from $\mtt$ (as introduced in \cite{qu12} sec.7.2) to the doctrine $\overline{\Prop}$, as represented on the left side of the following diagram.

\[
\begin{tikzcd}[row sep=large]
	{\mathcal{CM}\op} && {\CC_r\op} && {\mathcal{Q}_{G^\mathbf{mTT}}\op} && {\peff\op} \\
	& {\mathbf{Pos}} &&&& {\mathbf{Pos}}
	\arrow[from=1-1, to=1-3]
	\arrow[""{name=0, anchor=center, inner sep=0}, "{G^\mathbf{mTT}}"', from=1-1, to=2-2]
	\arrow[""{name=1, anchor=center, inner sep=0}, "{\overline{\mathbf{Prop}^r}}\cong \mathbf{wSub}_{\CC_r}", from=1-3, to=2-2]
	\arrow[from=1-5, to=1-7]
	\arrow[""{name=2, anchor=center, inner sep=0}, "{\overline{(G^\mathbf{mTT})}}"', from=1-5, to=2-6]
	\arrow[""{name=3, anchor=center, inner sep=0}, "{\peffpr\cong \mathbf{Sub}_\peff}", from=1-7, to=2-6]
	\arrow[shorten <=7pt, shorten >=7pt, from=0, to=1]
	\arrow[shorten <=8pt, shorten >=8pt, from=2, to=3]
\end{tikzcd}\]

Such a morphism lifts to a morphism between the corresponding elementary quotient completions, as represented on the right side of the above diagram. In the category $\mathcal{Q}_{G^{\mathbf{mTT}}}$ one can define an adequate notion of canonical isomorphisms, as done in \cite{m09}, thus producing a model for $\mathbf{emTT}$.
This, combined with the morphism above and the fibered structure of $\peff$ provided in this article, makes possible a direct interpretation of $\mathbf{emTT}$, also extended with inductive and coinductive predicates \cite{mmr21,mmr22,mfps23,phdthesisSabelli}, within $\peff$.


\section{Comparison with related works}
In \cite{VDB}, van den Berg and Moerdijk propose a construction of a predicative rendering of Hyland's Effective Topos 
in the wider context of algebraic set theory. Starting from a category with small maps $(\mathcal{E},\mathcal{S})$
they first construct a category of assemblies 
$\mathcal{A}sm_{\mathcal{E}}$ endowed with a family of maps induced on it by $\mathcal{S}$, and identified its exact completion   $(\mathcal{Eff}_{\mathcal{E}}, \mathcal{S}_{\mathcal{E}})$ as their notion of predicative realizability category (where
$\mathcal{Eff}_{\mathcal{E}}$ is a suitable subcategory of the exact on regular completion of  $\mathcal{A}sm_{\mathcal{E}}$).
We can compare our construction with theirs,
 as soon as we take for $\mathcal{E}$, the category  $\mathcal{E}[\mathbf{T}]$
 of definable classes in an extension $\mathbf{T}$ of $\mathbf{CZF}$ having the numerical existence property, endowed with the collection $\mathcal{S}[\mathbf{T}]$ of set-fibered functional relations between them (see sec.8 of \cite{VDBI}). Examples of such a theory $\mathbf{T}$ are $\mathbf{CZF}$ itself and  $\mathbf{CZF}+\mathbf{REA}$ as proven in \cite{rathjenexistence} and also
 $\mathbf{CZF}+\mathbf{RDC}+\bigcup-\mathbf{REA}$ 
 (by adequately modifying the proof of Theorem 7.2 in \cite{rathjenexistence} and combining it with Theorem 7.4 in \cite{RATAC}).

Over such a category, assemblies are pairs $(A,\alpha)$ in which $A$ is a definable class of $\mathbf{T}$ and $\alpha$ is a definable relation between elements of $A$ and of $\mathbb{N}$ such that $\mathbf{T}\vdash \forall x\in A\, \exists y\in \mathbb{N}\,\alpha(x,y)$.
Arrows between assemblies $(A,\alpha)$ and $(B,\beta)$ are equivalence classes of functional relations $F$ from $A$ to $B$ such that 
$$\mathbf{T}\vdash \exists e\in \mathbb{N}\,\forall x\in A\,\forall y\in B\,\forall z\in \mathbb{N}\,(\alpha(x,z)\wedge F(x,y)\rightarrow \beta(y,\{e\}(z)))$$
The assemblies category is denoted in \cite{VDB} as $\mathcal{Asm}_{\mathcal{E}[\mathbf{T}]}$.

If we now consider our category $\CC_r$ of realized collections in section.~\ref{reacol} formalized  in $ \mathbf{T}$, every object $\{x|\,\varphi(x)\}$ (for which $\mathbf{T}\vdash \forall x\,( \varphi(x)\rightarrow x\in \mathbb{N})$ always holds) can be mapped to the assembly $\mathbf{i}(\{x|\,\varphi(x)\}):=(\{x|\varphi(x)\},=_\varphi)$ where $$x=_\varphi y\equiv^{def} \varphi(x)\wedge \varphi(y)\wedge x=y$$
Arrows $[\mathbf{n}]_\approx$ in $\CC_{r}$ can be sent to the functional relations representing their graph, thus defining a functor $\mathbf{i}$ from $\CC_r$ to $\mathcal{Asm}_{\mathcal{E}[\mathbf{T}]}$. Since the numerical existence property holds for $\mathbf{T}$, the functor $\mathbf{i}$ is full and faithful, thus $\CC_r$ can be identified as a full subcategory of  $\mathcal{Asm}_{\mathcal{E}[\mathbf{T}]}$.
Now, let us call \textit{recursive objects} those  assemblies
where $A$ is a subclass of natural numbers satisfying
$$\mathbf{T}\vdash \forall x\in A\,\forall y\in \mathbb{N}\,(\alpha(x,y)\leftrightarrow x=y)$$
If we call
 $\mathcal{Rec}_{\mathcal{E}[\mathbf{T}]}$ the full subcategory of  assemblies with recursive objects, then $\CC_r$ can be seen as a full subcategory of $\mathcal{Rec}_{\mathcal{E}[\mathbf{T}]}$. 
$$\xymatrix{\CC_r\ar@{^{(}->}^{\mathbf{i}}[rr]\ar@{^{(}->}[rd]        &           &\mathcal{Asm}_{\mathcal{E}[\mathbf{T}]}\\
    &\mathcal{Rec}_{\mathcal{E}[\mathbf{T}]} \ar@{^{(}->}_{\mathbf{incl}}[ru]  &\\
}$$

Moreover,  recalling that the ex/lex completion of the category of recursive objects turns out to be that of { discrete objects} in \cite{rosolinidiscrete,vanoostenhomotopy},  we denote the ex/lex completion of $\mathcal{Rec}_{\mathcal{E}[\mathbf{T}]}$ with $\mathcal{Disc}_{\mathcal{E}[\mathbf{T}]}$. Finally,  since $\CC_r$ is a full subcategory of {recursive objects}
 and $\peff$ is the ex/lex completion of $\CC_r$ we have an embedding of $\peff$ in $\mathcal{Disc}_{\mathcal{E}[\mathbf{T}]}$ that extends to fibers, trivially.

\[
\begin{tikzcd}[column sep= tiny]
	{\peff^\to_{set}} && {\peff^\to} && {\mathcal{Disc}_{\mathcal{E}[\mathbf{T}]}^\to} \\
	& \peff && {\mathcal{Disc}_{\mathcal{E}[\mathbf{T}]}}
	\arrow[hook, from=1-1, to=1-3]
	\arrow[from=1-1, to=2-2]
	\arrow[from=1-3, to=1-5]
	\arrow[from=1-3, to=2-2]
	\arrow[from=1-5, to=2-4]
	\arrow[from=2-2, to=2-4]
\end{tikzcd}\]

In addition, recalling that $\mathcal{Disc}_{\mathcal{E}[\mathbf{T}]}$ is the ex/reg completion of the category of modest sets (see \cite{pinomod}), namely the
category of assemblies which satisfy
$$\mathbf{T}\vdash \forall x,y\in A\forall z\in \mathbb{N}(\alpha(x,z)\wedge \alpha(y,z)\rightarrow x=y)$$ after denoting
such a category  within $\mathcal{Eff}_{\mathcal{E}[\mathbf{T}]}$ with $\mathcal{Mod}_{\mathcal{E}[\mathbf{T}]}$  we obtain the following result.

\begin{prop} Let $\peff$ be the predicative effective topos built on an extension  $\mathbf{T}$ of $\mathbf{CZF}$ having the numerical existence property; and let  $\peff_{sm.\Delta}$ be the full subcategory of $\peff$  given by objects $(\Aa,[\RR])$ where $[\RR]\in \Props(\Aa \times \Aa)$. Then, $\peff_{sm.\Delta}$  is a full {subcategory} of the predicative realizability category $\mathcal{Eff}_{\mathcal{E}[\mathbf{T}]}$ defined as in \cite{VDB} from the syntactic category $\mathcal{E}[\mathbf{T}]$ of definable classes of $\mathbf{T}$ endowed with the family $\mathcal{S}[\mathbf{T}]$ of arrows whose fibers are sets in $\mathbf{T}$.

\[
\begin{tikzcd}
	{\peff_{sm.\Delta}} && {\mathcal{Eff}_{\mathcal{E}[\mathbf{T}]}} \\
	{\peff\cong (\CC_r)_{ex/lex}} & {\mathcal{Disc}_{\mathcal{E}[\mathbf{T}]}=(\mathcal{Mod}_{\mathcal{E}[\mathbf{T}]})_{ex/reg}} & {(\mathcal{Asm}_{\mathcal{E[\mathbf{T}]}})_{ex/reg}}
	\arrow[hook, from=1-1, to=1-3]
	\arrow[hook, from=1-1, to=2-1]
	\arrow[hook, from=1-3, to=2-3]
	\arrow[hook, from=2-1, to=2-2]
	\arrow[hook, from=2-2, to=2-3]
\end{tikzcd}\]
\begin{proof}
The category $\mathcal{Eff}_{\mathcal{E}[\mathbf{T}]}$ proposed in \cite{VDB} as a predicative version of Hyland's Effective Topos $\eff$ is the full subcategory of the ex/reg completion of $\mathcal{Asm}_{\mathcal{E}[\mathbf{T}]}$, whose objects are 
\textit{separated} in the sense of pg.\ 141 of \cite{VDBI}. The statement follows observing that objects $(\Aa,[\RR])$ of $\peff_{sm.\Delta}$, seen as objects in $(\mathcal{Asm}_{\mathcal{E[\mathbf{T}]}})_{ex/reg}$, are separated in the above sense since $[\RR]\in \Props(\Aa\times \Aa)$.
\end{proof}
\end{prop}

Observe that, when we consider $\mathbf{T}$ to be $\mathbf{ZFC}$, although it does not satisfy the numerical existence property, we can still define a functor from $\peff$ to a version $\eff_{def}$ of Hyland's Effective Topos  built as the ex/reg completion of the full subcategory of $\mathcal{Asm}_{\mathcal{E}[\mathbf{ZFC}]}$ whose objects are assemblies of which the underlying classes are definable sets. 
 Indeed, in $\mathbf{ZFC}$ all definable subclasses of $\mathbb{N}$ are in fact sets.
In the same framework, one can see that this $\eff_{def}$ is a full subcategory of $\bf \mathcal{Eff}_{\mathcal{E}[\mathbf{ZFC}]}$.

$$\xymatrix{\peff\ar@{^{(}->}^{not\  full}[rr]   & &\eff_{def} \ar@{^{(}->}[r]   &\mathcal{Eff}_{\mathcal{E}[\mathbf{ZFC}]}}$$

A comparison between $\peff$ formalized in $\tar$, and  the usual  Hyland's Effective Topos  $\eff$\ as an internal large category within {\bf ZFC} is induced by the standard interpretation of $\tar$ in the {\bf ZFC} set theory, as shown in \cite{maietti_maschio_2021}.


All these observations have taken into account only the base category $\peff$. A comparison between their fibered structures is left to future work.  Just observe that while maps in van den Berg and Moerdijk 's approach are built on the notion of set in axiomatic set theory, in our approach the fibration of sets is defined mimicking the structure of sets within the language of dependent type theory as that used in both levels of the Minimalist Foundation \cite{m09}.

Nevertheless,  our comparisons are sufficient to conclude that {\it the full subcategory of discrete objects of  Hyland's Effective Topos  or  of van den Berg and Moerdijk's constructive rendering contains already a predicative version of $\mathbf{Eff}$}.  In a few words, here  we confirm what said on p. 1140  of \cite{vanoostenhomotopy} by concluding that {\it the discrete objects of $\mathbf{Eff}$ can be considered the realm of constructive predicative mathematics}.


\section{Conclusions}
We have shown here that {\it the  full subcategory of discrete objects of Hyland's Effective topos $\mathbf{Eff}$}\ contains already a {\it fibred predicative topos} validating the formal Church's thesis,  even when both are formalized in a constructive metatheory.  We did so by enucleating  a fibred structure for $\peff$ introduced in
\cite{maietti_maschio_2021} that would be enough to model both levels of \mf\ and its extensions with inductive and conductive predicates in \cite{mmr21,mmr22,mfps23,phdthesisSabelli}.  

Furthermore, thanks to the equiconsistency of both levels of \mf\ with their classical version shown in \cite{mfequicla}, we could use $\peff$ as a computational model for the classical version of \mf, too. And,  since  both levels of \mf\ can be interpreted in $\mathbf{HoTT}$\ in a compatible way
as shown  in \cite{contentemaietti},   it would be interesting to investigate whether  $\peff$  is  somewhat  related to computational models of  Homotopy Type theory,  for example  to that in \cite{norcub}.

In the future, we intend to investigate   abstract notions of {\it fibred predicative topos} and of {\it predicative tripos-to-topos construction} of which  the fibred structure of $\peff$   enucleated here  would provide a key example.

\textbf{Acknowledgements}
The authors would like to acknowledge Michael Rathjen, Pino Rosolini, Pietro Sabelli, Thomas Streicher  
and Davide Trotta for fruitful discussions.

All the authors are affiliated with the INdAM national research group GNSAGA. Moreover, the first and second authors would like to thank the Hausdorff Research Institute for Mathematics for the financial support received to attend part of the trimester program 'Prospects of Formal Mathematics' when significant progress was made for this work.

\bibliography{bib2}

\begin{thebibliography}{IMMS18}

\bibitem[AR01]{czf}
P.~{Aczel} and M.~{Rathjen}.
\newblock Notes on constructive set theory.
\newblock Mittag-Leffler Technical Report No.40, 2001.

\bibitem[Awo18]{awodeynaturalmodels}
Steve Awodey.
\newblock Natural models of homotopy type theory.
\newblock {\em Mathematical Structures in Computer Science}, 28(2):241--286,
  2018.

\bibitem[BM08]{VDBI}
B.~van~den {B}erg and I.~{Moerdijk}.
\newblock Aspects of predicative algebraic set theory {I}: Exact completion.
\newblock {\em Annals of Pure and Applied Logic,}, 156(1):123--159, 2008.

\bibitem[BM11]{VDB}
B.~van~den {Berg} and I.~{Moerdijk}.
\newblock Aspects of predicative algebraic set theory {II}: Realizability.
\newblock {\em Theoretical Computer Science}, 412:1916--1940, 2011.

\bibitem[Car95]{car}
A.~Carboni.
\newblock Some free constructions in realizability and proof theory.
\newblock {\em Journal of Pure and Applied Algebra}, 103:117--148, 1995.

\bibitem[Cio23]{cioffo2023biased}
C.~J. Cioffo.
\newblock Biased elementary doctrines and quotient completions.
\newblock {\em arXiv preprint arXiv:2304.03066}, 2023.

\bibitem[CLW93]{carboniextensive}
A.~Carboni, S.~Lack, and R.~F.~C. Walters.
\newblock Introduction to extensive and distributive categories.
\newblock {\em Journal of Pure and Applied Algebra}, 84(2):145--158, 1993.

\bibitem[CM82]{carboniceliamagno}
A.~Carboni and R.~Celia Magno.
\newblock The free exact category on a left exact one.
\newblock {\em J. Austral. Math. Soc. Ser. A}, 33(3):295--301, 1982.

\bibitem[CM24]{contentemaietti}
M.~Contente and M.E. Maietti.
\newblock The compatibility of the minimalist foundation with homotopy type
  theory.
\newblock {\em Theoretical Computer Science}, 991:114421, 2024.

\bibitem[CV98]{carbonivitale}
A.~Carboni and E.~M. Vitale.
\newblock Regular and exact completions.
\newblock {\em Journal of Pure and Applied Algebra}, 125(1-3):79--116, 1998.

\bibitem[Dyb96]{categorieswithfamilies}
P.~Dybjer.
\newblock Internal type theory.
\newblock In {\em Types for proofs and programs ({T}orino, 1995)}, volume 1158
  of {\em Lecture Notes in Comput. Sci.}, pages 120--134. Springer, Berlin,
  1996.

\bibitem[Fef82]{feferman}
S.~Feferman.
\newblock Iterated inductive fixed-point theories: application to hancock's
  conjecture.
\newblock In {\em Studies in Logic and the Foundations of Mathematics}, volume
  109, pages 171--196. Elsevier, 1982.

\bibitem[HJP80]{tripos}
J.~M.~E. {H}yland, P.~T. {J}ohnstone, and A.~M. {P}itts.
\newblock Tripos theory.
\newblock {\em Bulletin of the Australian Mathematical Society}, 88:205--232,
  1980.

\bibitem[Hof04]{hofstra2004relative}
P.~J. Hofstra.
\newblock Relative completions.
\newblock {\em Journal of Pure and Applied Algebra}, 192(1-3):129--148, 2004.

\bibitem[HRR90]{rosolinidiscrete}
J.~M.~E. Hyland, E.~P. Robinson, and G.~Rosolini.
\newblock The discrete objects in the effective topos.
\newblock {\em Proc. London Math. Soc. (3)}, 60(1):1--36, 1990.

\bibitem[{Hyl}82]{eff}
J.~M.~E. {Hyland}.
\newblock The effective topos.
\newblock In {\em The L.E.J. Brouwer Centenary Symposium (Noordwijkerhout,
  1981)}, volume 110 of {\em Stud. Logic Foundations Math.}, pages 165--216.
  North-Holland, Amsterdam-New York,, 1982.

\bibitem[IMMS18]{IMMSt}
H.~Ishihara, M.~E. Maietti, S.~Maschio, and T.~Streicher.
\newblock Consistency of the intensional level of the minimalist foundation
  with church’s thesis and axiom of choice.
\newblock {\em Archive for Mathematical Logic}, 57(7-8):873--888, 2018.

\bibitem[Jac99]{jacobsbook}
B.~Jacobs.
\newblock {\em Categorical logic and type theory}, volume 141 of {\em Studies
  in Logic and the Foundations of Mathematics}.
\newblock North-Holland Publishing Co., Amsterdam, 1999.

\bibitem[JT94]{facetI}
G.~Janelidze and W.~Tholen.
\newblock Facets of descent. {I}.
\newblock {\em Applied Categorical Structures}, 2(3):245--281, 1994.

\bibitem[Mai05]{maiettimodular}
M.~E. Maietti.
\newblock Modular correspondence between dependent type theories and categories
  including pretopoi and topoi.
\newblock {\em Math. Structures Comput. Sci.}, 15(6):1089--1149, 2005.

\bibitem[{Mai}09]{m09}
M.~E. {Maietti}.
\newblock A minimalist two-level foundation for constructive mathematics.
\newblock {\em Annals of Pure and Applied Logic,}, 160(3):319--354, 2009.

\bibitem[MM12]{maclane2012sheaves}
S.~MacLane and I.~Moerdijk.
\newblock {\em Sheaves in geometry and logic: A first introduction to topos
  theory}.
\newblock Springer Science \& Business Media, 2012.

\bibitem[MM16]{MM2}
M.~E. Maietti and S.~Maschio.
\newblock A predicative variant of a realizability tripos for the {M}inimalist
  {F}oundation.
\newblock {\em IfColog Journal of Logics and their Applications},
  3(4):595--668, 2016.

\bibitem[MM21]{maietti_maschio_2021}
M.~E. Maietti and S.~Maschio.
\newblock A predicative variant of {H}yland's effective topos.
\newblock {\em The Journal of Symbolic Logic}, 86(2):433–447, 2021.

\bibitem[MMR21]{mmr21}
M.~E. Maietti, S.~Maschio, and M.~Rathjen.
\newblock A realizability semantics for inductive formal topologies, church's
  thesis and axiom of choice.
\newblock {\em Logical Methods in Computer Science}, 17, 2021.

\bibitem[MMR22]{mmr22}
M.~E. Maietti, S.~Maschio, and M.~Rathjen.
\newblock Inductive and coinductive topological generation with church's thesis
  and the axiom of choice.
\newblock {\em Logical Methods in Computer Science}, 18, 2022.

\bibitem[MP02]{moerdijkpalmgren2002type}
I.~Moerdijk and E.~Palmgren.
\newblock Type theories, toposes and constructive set theory: predicative
  aspects of {AST}.
\newblock {\em Annals of Pure and Applied Logic,}, 114(1-3):155--201, 2002.

\bibitem[MR13a]{elqu}
M.~E. Maietti and G.~Rosolini.
\newblock Elementary quotient completion.
\newblock {\em Theory and Applications of Categories}, 27(17):445--463, 2013.

\bibitem[MR13b]{qu12}
M.~E. {Maietti} and G.~{Rosolini}.
\newblock Quotient completion for the foundation of constructive mathematics.
\newblock {\em Logica Universalis}, 7(3):371--402, 2013.

\bibitem[MR15]{uec}
M.~E. Maietti and G.~Rosolini.
\newblock Unifying exact completions.
\newblock {\em Appl. Categ. Structures}, 23(1):43--52, 2015.

\bibitem[MR16]{MR16}
M.~E. {Maietti} and G.~{Rosolini}.
\newblock Relating quotient completions via categorical logic.
\newblock In Dieter Probst and Peter Schuster, editors, {\em Concepts of Proof
  in Mathematics, Philosophy, and Computer Science}, pages 229--250, 2016.

\bibitem[MS05]{mtt}
M.~E. {Maietti} and G.~{Sambin}.
\newblock {Toward a minimalist foundation for constructive mathematics}.
\newblock In {L. Crosilla and P. Schuster}, editor, {\em From Sets and Types to
  Topology and Analysis: Practicable Foundations for Constructive Mathematics},
  number~48 in {Oxford Logic Guides}, pages 91--114. {Oxford University Press},
  2005.

\bibitem[MS23]{mfps23}
M.~E. Maietti and P.~Sabelli.
\newblock A topological counterpart of well-founded trees in dependent type
  theory.
\newblock In Marie Kerjean and Paul~Blain Levy, editors, {\em Proceedings of
  the 39th Conference on the Mathematical Foundations of Programming Semantics,
  {MFPS} XXXIX, Indiana University, Bloomington, IN, USA, June 21-23, 2023},
  volume~3 of {\em {EPTICS}}. EpiSciences, 2023.

\bibitem[MS24]{mfequicla}
M.E. Maietti and P.~Sabelli.
\newblock Equiconsistency of the minimalist foundation with its classical
  version.
\newblock {\em Annals of Pure and Applied Logic,}, 2024.
\newblock to appear.

\bibitem[MT21]{maiettitrottageneralizedexistentialcompletionsregular}
M.~E. Maietti and D.~Trotta.
\newblock Generalized existential completions and their regular and exact
  completions, 2021.

\bibitem[MT22]{chgenexco}
M.~E. Maietti and D.~Trotta.
\newblock A characterization of generalized existential completions.
\newblock {\em Ann. Pure Appl. Log.}, 174:103234, 2022.

\bibitem[Rat05]{rathjenexistence}
M.~Rathjen.
\newblock The disjunction and related properties for constructive
  {Z}ermelo-{F}raenkel set theory.
\newblock {\em The Journal of Symbolic Logic}, 70(4):1232--1254, 2005.

\bibitem[Rat08]{RATAC}
M.~Rathjen.
\newblock Metamathematical properties of intuitionistic set theories with
  choice principles.
\newblock In {\em New computational paradigms}, pages 287--312. Springer, New
  York, 2008.

\bibitem[Ros90]{pinomod}
G.~Rosolini.
\newblock About modest sets.
\newblock volume~1, pages 341--353. 1990.
\newblock Third Italian Conference on Theoretical Computer Science (Mantova,
  1990).

\bibitem[RR90]{robinsonrosolini}
E.~Robinson and G.~Rosolini.
\newblock Colimit completions and the effective topos.
\newblock {\em The Journal of Symbolic Logic}, 55(2):678--699, 1990.

\bibitem[SA21]{norcub}
J.~Sterling and C.~Angiuli.
\newblock Normalization for cubical type theory.
\newblock In {\em 2021 36th Annual ACM/IEEE Symposium on Logic in Computer
  Science (LICS)}, pages 1--15, 2021.

\bibitem[Sab24]{phdthesisSabelli}
P.~Sabelli.
\newblock {\em Around the Minimalist FOundation: (Co)Induction and
  Equiconsistency}.
\newblock PhD thesis, Università degli Studi di Padova, 2024.

\bibitem[vO94]{vanOOsten}
J.~van Oosten.
\newblock Axiomatizing higher order {K}leene realizability.
\newblock {\em Annals of Pure and Applied Logic,}, 70(87-111), 1994.

\bibitem[vO08]{VOO08}
J.~van Oosten.
\newblock {\em Realizability: an introduction to its categorical side}, volume
  152 of {\em Studies in Logic and Foundations of Mathematics}.
\newblock Elsevier, 2008.

\bibitem[vO15]{vanoostenhomotopy}
J.~van Oosten.
\newblock A notion of homotopy for the effective topos.
\newblock {\em Math. Structures Comput. Sci.}, 25(5):1132--1146, 2015.

\end{thebibliography}
\bibliographystyle{alpha}

\end{document}